\documentclass[12pt]{article}

\input{amssym.def}
\newtheorem{theorem}{\bf Theorem}[section]
\newtheorem{corollary}[theorem]{\bf Corollary}
\newtheorem{proposition}[theorem]{\bf Proposition} 
\newtheorem{notation}[theorem]{\sc Notation} 
\newtheorem{definition}[theorem]{\bf Definition} 
\newtheorem{lemma}[theorem]{\bf Lemma} 
\newtheorem{example}[theorem]{\it Example}  
\newtheorem{examples}[theorem]{\it Examples}  
\newtheorem{algorithm}[theorem]{\bf Algorithm} 
\newtheorem{conjecture}[theorem]{\bf  Conjecture}
\newtheorem{remark}[theorem]{\it Remark}
\newtheorem{remarks}[theorem]{\it  Remarks}

\newenvironment{proof}{{\noindent\it Proof.\ }}{$\bullet$\par\vspace{4mm}} 
\def \bt{ \begin{theorem} }
\def \et{ \end  {theorem} }
\def \bl{ \begin{lemma} }
\def \el{ \end  {lemma} }
\def \bp{ \begin{proposition} }
\def \ep{ \end  {proposition} }
\def \bn{ \begin{notation} }
\def \en{ \end  {notation} }
\def \bq{ \begin {question} }
\def \eq{ \end {question} }
\def \bc{ \begin{corollary} }
\def \ec{ \end  {corollary} }
\def \bcj{ \begin{conjecture} }
\def \ecj{ \end  {conjecture} }
\def \bd{ \begin{definition} }
\def \ed{ \end  {definition} }
\def \bdp{ \begin{definitionprop} }
\def \edp{ \end  {definitionprop} }
\def \bdt{ \begin{definitiontheorem} }
\def \edt{ \end  {definitiontheorem} }
\def \bpr{ \begin  {proof} }
\def \epr{ \end  {proof} }
\def \ba{ \begin{algorithm} }
\def \ea{ \end{algorithm} }
\def \be{ \begin{example} }
\def \eex{ \end{example} }
\def \bes{ \begin{examples} }
\def \eexs{ \end{examples} }
\def \br{ \begin{remark} }
\def \er{ \end{remark} }
\def \brs{ \begin{remarks} }
\def \ers{ \end{remarks} }
\def \bpb{ \begin{problem} }
\def \epb{ \end{problem} }

\newcommand{\lcm} {\mathrm{lcm}}

\newcommand{\G}{\Gamma}   
\newcommand{\BFn}{\hbox{\bf n} }   
\newcommand{\BFi}{\hbox{\bf i} } 
\newcommand{\sBFi}{\hbox{\bf \scriptsize i}} 
\newcommand{\BFZ}{\hbox{\bf 0} }  
\newcommand{\BFone}{\hbox{\bf 1}}  
\newcommand{\sBFone}{\hbox{\bf \scriptsize 1}}   
\newcommand{\BFX}{\hbox {\bf\it X} }  
\begin{document}
\begin{verbatim}
In "J.Symbolic Computation (1995), Vol. 20, 71-92."

\end{verbatim}
\begin{center}
ON $n$--DIMENSIONAL SEQUENCES. I 
\footnote{Research supported by Science and
Engineering Research Council Grant GR/H15141.

Current Address: School of Mathematics and Physics, University of Queensland, Brisbane, Queensland 4072, Australia, Email: ghn@maths.uq.edu.au, ORCID 0000-0003--518-5833.)}
\end{center}
\begin{center}
GRAHAM NORTON

Centre for Communications Research, University of Bristol, England.

Dedicated to Professor Peter J. Hilton on his retirement.
\end{center}
\begin{abstract} 
Let $R$ be a commutative ring and let $n \geq 1.$
We study $\Gamma(s)$, the generating function and Ann$(s)$, the ideal of
characteristic polynomials of $s$, an $n$--dimensional sequence over $R$.\\

We express $f(X_1,\ldots,X_n) \cdot \G(s)(X_1^{-1},\ldots ,X_n^{-1})$
as a partitioned sum. That is, we give (i) a $2^n$--fold ``border'' 
partition (ii) an explicit expression for the product 
as a $2^n$--fold sum; the support of each summand is contained in 
precisely one member of the partition. A key summand is $\beta_0(f,s)$,
the ``border polynomial'' of $f$ and $s$, which is divisible by 
$X_1\cdots X_n$.\\

We say that $s$ is {\em eventually rectilinear} if the elimination ideals
Ann$(s)\cap R[X_i]$ contain an $f_i(X_i)$ for $1 \leq i \leq n$. 
In this case, we show that $\mbox{Ann}(s)$ is the ideal quotient 
$(\sum_{i=1}^n(f_i)\ :\ \beta_0(f,s)/(X_1\cdots X_n)).$\\ 

When $R$ and $R[[X_1,X_2, \ldots ,X_n]]$ are factorial domains
(e.g. $R$ a principal ideal domain or ${\Bbb F}[X_1,\ldots,X_n]$), 
we compute {\em the monic generator} $\gamma _i$ of $\mbox{Ann}(s) 
\cap R[X_i]$ from known $f_i \in \mbox{Ann}(s) \cap R[X_i]$ or 
from a finite number of $1$--dimensional 
linear recurring sequences over $R$. Over a field ${\Bbb F}$ this gives an 
$O(\prod_{i=1}^n \delta \gamma _i^3)$ algorithm to compute an 
${\Bbb F}$--basis for $\mbox{Ann}(s)$.\\ 
\end{abstract}


\section{Introduction}
Linear recurring sequences (lrs) have a long and useful history; 
see for example 
Cerlienco {\em et al.} (1987), Zierler (1959) and the works cited there. 
Sequences over the integers and over finite fields are applied {\em 
inter alia} in the Analysis of Algorithms (Greene \& Knuth (1982)), 
Telecommunications (McEliece (1987)), Coding Theory 
(Peterson \& Weldon (1972), Fitzpatrick \& Norton (1991)) 
and Cryptography (Rueppel (1986)). 

Applications of more general sequences have also appeared in the recent
literature: sequences over ${\Bbb Z}/4{\Bbb Z}$ (Bozta\c{s}, 
Hammons and Kumar (1992)), over the Gaussian integers (Fan \& Darnell 
(1994))
and $2$--dimensional lrs over a field (i.e. lrs indexed by 
${\Bbb N}^2$); see for example Homer \& Goldmann 
(1985) and the works cited there, Lin \& Liu (1988), Prabhu \& Bose (1982) 
and Sakata (1978, 1981, 1990), Fitzpatrick \& Norton (1990), 
Chabanne \& Norton (1994)), Norton (1995b).

We consider $n$--dimensional ($n$--D) 
sequences over an arbitrary commutative ring $R$, where $n \geq 1$.
In Section 2 we set up our basic framework and develop some
preliminary properties of (i) the ideal Ann$(s)$ of characteristic 
polynomials and (ii) $\Gamma(s)$, the generating function of an $n$--D
sequence $s$. Section 3 begins with a certain ``border'' partition and studies
the resulting decomposition of the product 
$f(X_1,\ldots,X_n)\cdot\Gamma(s)(X_1^{-1},\ldots,X_n^{-1})$.
Finally, some properties of $\Gamma(s)$ and Ann$(s)$ for an ``eventually 
rectilinear (EVR)'' sequence are given in Section 4.

In more detail, we study sequences over $R$ indexed {\em negatively} i.e. 
by $-{\Bbb N}^n$. 
The traditional approach to studying $\Gamma(s)$ and Ann$(s)$ is 
to index sequences by ${\Bbb N}^n$ and to  
regard $\Gamma (s)$ as an element of the power series ring
$R[[X_1,\ldots,X_n]]$; in this way, 
$R[X_1,\ldots,X_n]$ acts by {\em left}--shifting. 
This approach complicates
the theory and proofs unnecessarily however.
For example, it forces use of the 
reciprocal $f^\ast$ of a polynomial $f$ (for $f^\ast\Gamma(s)$ to be
in the power series ring), resulting in  
a characteristic polynomial $f$ being characterized in terms of 
$\Gamma (s)$, its degree and $f^\ast$. Compare also Corollaries 2.9,
2.12 and 3.9, 3.10 below.

Instead, we let $f\circ s$ denote $f \in R[X_1,\ldots,X_n]$ acting on $s$ 
by {\em right--shifting} (defined in Section 2) and let 
$\mbox{Ann}(s)$ denote the annihilator ideal of $s$ with respect to 
right--shifting. We use the ring of Laurent series 
$R((X_1^{-1},\ldots,X_n^{-1}))$ --- rather than power series ---
as the ambient ring. This is a much more natural approach: the ring of 
Laurent series contains $R[X_1,\ldots,X_n]$ 
{\em as standard $R[X_1,\ldots,X_n]$--submodule}, as well as 
the product $f\cdot\Gamma(s)$. 

Section 2 defines a partial order on ${\Bbb Z}^n$ and generalizes 
standard interval notation to ${\Bbb Z}^n$. We give the basic definitions, 
various examples of $n$--D lrs, define EVR sequences and give some
preliminary results relating Ann$(s)$ and $\Gamma(s).$ For example, we show
that $\Gamma(f\circ s)$ is a summand of $f\cdot\Gamma(s).$

In Section 3, we define the border partition and the ``border polynomial''
$\beta_0(f,s)$ of $f$ and $s$, a key summand of $f\cdot\Gamma(s)$, which
is divisible by $X_1 \cdots X_n$.
The remaining $2^n-2$ summands are ``border Laurent series''. 
We also show how these border summands can be
rewritten using a generalized Newton divided difference operator.
Section 3 concludes with our decomposition formula for $f\cdot \Gamma(s).$

Section 4 concentrates on EVR sequences over commutative rings, 
characterizing their generating functions and exhibiting Ann$(s)$ as an ideal 
quotient. This requires looking at the elimination ideals 
$\mbox{Ann}(s) \cap R[X_i], \ (1 \leq i \leq n)$
via some ancillary results on 1--D sequences and 1--D ``associated 
sequences''.

We call $R$ {\em potential} if for all $n \geq 1, R[[X_1,X_2,\ldots ,X_n]]$ is 
a factorial --- unique factorisation --- domain. In this case, we show that 
the above elimination ideals have a unique {\em monic generator}, and give
two ways of finding them: Theorem~\ref{Annsi} and Corollary~\ref{Annlcm}. 
We also generalize a result of Cerlienco \& Piras (1991) to sequences 
over potential domains (Theorem~\ref{cofinite}).

Finally, we show that if $s$ is EVR with $f_i \in \mbox{Ann}(s)\cap R[X_i]$,
then Ann$(s)$ is the ideal quotient 
$(\sum_{i=1}^n(f_i)\ :\ \beta_0(f,s)/(X_1 \cdots X_n)).$
Algorithm~\ref{Anness} below shows how to compute 
the corresponding basis for $\mbox{Ann}(s)$.
There are known methods (Buchberger (1985) and Cox {\em et al.} (1991)) 
for finding a 
basis of an ideal quotient. The former yields a linear system of size 
$\prod_{i=1}^n \delta \gamma _i$ and the latter yields a groebner basis for 
$\mbox{Ann}(s)$ (with respect to the lexicographic term order). 

Our original goals were two--fold: to try and simplify the Commutative Algebra
used in some early papers on the structure of 2--D cyclic codes (Imai \&
Arakaki (1974), Ikai, Kosako \& Kojima (1975, 1976), Imai (1977), 
Sakata (1978, 1981)) and secondly, to better understand Sakata (1988). It 
seemed that this would be helped by undertaking a fundamental study of $n$--D
sequences {\em per se}. 
(Even though our principal application is the case $n=2$, we 
found that formulating and proving our results {\em for all} $n \geq 1$ 
improved and simplified the theory.)
For some preliminary applications of this work to $n$--D 
cyclic codes and their
duals, see Norton (1995b). An application to simplified 
encoding of $n$--D cyclic codes was presented in Norton (1995c). 

We would suggest that the use of right--shifting and 
$R((X_1^{-1},\ldots,X_n^{-1}))$ sharpens and simplifies the theory of 
both finite and linear recurring sequences: (i) Theorem~\ref{main} 
generalizes the treatment of rectilinear sequences over a field 
studied in Fitzpatrick \& Norton (1990) to EVR sequences over an arbitrary
commutative ring and does
not require the ``Reduction Lemma'' of Fitzpatrick \& Norton (1990)
(ii) a {\em division--free} analogue of the Berlekamp--Massey algorithm 
in $R((X^{-1}))$, which computes a ``minimal realization'' $(f,\beta_0(f,s))$
of a finite sequence $s$ over a domain $R$ appears in Norton (1995a), 
together with some new applications of minimal realization
(iii) a theory of division--free minimal realizations of a finite $n$--D 
sequence over a domain will appear in Norton (1995e); 
{\em cf.} Sakata (1988, 1990). 

The results of this paper have been applied in Chabanne \& Norton (1994). In 
fact, the $n$--D key equation ({\em loc. cit.}, Theorem 2.4) is a special case 
of Corollary~\ref{Annsigen} over a finite field: $f_i$ is the error--locator 
$X_i$--polynomial and $\beta _0(\prod_{i=1}^n f_i ,s)/(X_1\cdots X_n)$ is 
the error--evaluator polynomial. Indeed, the minimal realization algorithm
of Norton (1995a) applied to the 1--D version of our
key equation gives a simpler way of
decoding a $t$--error correcting Reed--Solomon code in at most
$t(10t+1)$ multiplications, without using Forney's procedure to compute
the error magnitudes for example, Norton (1995d). 
It seems likely that our work will also apply to decoding geometric Goppa 
codes. 

Finally, our approach was partly suggested by the formulation for 
$1$--D sequences given in Ferrand (1988), which uses $R((X))$ as the ambient 
ring. We note that ${\Bbb F}((X^{-1}))$, where ${\Bbb F}$ denotes a field, 
was used in Welch \& Scholtz 
(1979) in the context of decoding BCH codes, and in Niedereiter (1988).

\subsection{Notation}  
In general, we use lower case Roman letters for elements, Greek letters for
functions and short names for sets. ${\Bbb N} = \{0,1,2,\ldots\}$ and 
${\Bbb Z} = \{0, \pm 1, \pm 2, \ldots \}$.

\begin{tabbing}
{\bf Notation} \= \hspace{1cm} \= {\bf Meaning}\\\\

$\BFn$ \> \> $\{1,2,\ldots ,n \}$.\\

$\pi _i$ \> \> Projection of ${\Bbb Z}^n$ onto the $i^{th}$ component, 
$i \in \BFn$.\\
 
$\BFX^a$ \> \> Monomial $X_1^{a_1}X_2^{a_2}\ldots X_n^{a_n}$, where 
$a_i \in {\Bbb Z}, i \in \BFn$.\\
 
$\widehat{\bf X}_i$ \> \> Variables $X_1, X_2, \ldots ,X_n$, excluding $X_i , 
i \in \BFn$.\\
 
$P_n$        \> \> Polynomials over $R$ in $X_1,\ldots,X_n$\\

$L_n$        \> \>  Laurent series over $R$ in $X_1^{-1},\ldots,X_n^{-1}.$\\

$S_n$        \> \> Power series over $R$ in $X_1^{-1},\ldots,X_n^{-1}.$\\

$\delta _ig$ \> \> Degree of $g \in  L_n$ as Laurent series 
in $X_i^{-1}, i \in \BFn$.\\
 
$\delta g$ \> \> Degree of $g \in L_n$ i.e. an $n$--vector.\\
 
 
$S^n(R)^-$ \> \> $n$--dimensional sequences over $R$, indexed by $-{\Bbb 
N}^n$.\\
 
$f\circ s$ \> \> Polynomial $f$ acting on sequence $s$ by shifting.\\
 
$\beta _0(f,s)$ \> \> Border polynomial of $f$ and $s$.\\
 
$\mbox{Ann}(s)$ \> \> Characteristic ideal of $s$.\\
 
$\gamma (s)$ \> \> Primitive generator of $\mbox{Ann}(s)$, $s \in S^1(R)^-, R$ 
factorial.\\
 
$s^{(i,\ast)}$ \> \> $s$ regarded as an element of 
$S^1(R[[\widehat{\bf X}_i^{-\hbox{\bf \scriptsize 1}}]])^-$.\\
 
$\Gamma (s)$ \> \> Generating function of $s$, as element of $S_n.$\\
\end{tabbing}
 
We equate sums and products over the empty set to $0$ and $1$ respectively.

\section{Preliminaries}
\subsection{ ${\Bbb Z}^n$ }
Throughout the paper, $n \geq 1$ and $\BFn = \{1,2,\ldots ,n \}$. 
Addition and negation in ${\Bbb Z}^n$ are componentwise; thus we can write 
$-{\Bbb N}^n = (-{\Bbb N})^n.$ We let $\BFZ, \BFone$ be the 
points in ${\Bbb Z}^n$ with all components $0, 1$ respectively. 

For $i \in \BFn,\ \pi _i : {\Bbb Z}^n \rightarrow {\Bbb Z}$ 
is the projection onto the $i^{th}$ factor; for $a \in {\Bbb Z}^n$,  
we also write $a_i$ for $\pi _ia$.

${\Bbb Z}^n$ is {\em partially ordered} by the relation $\leq$ 
on each component: $a \leq b \mbox{ iff } a_i \leq b_i$ {\em 
for all} $i \in \BFn$; $a \geq b$ is synonymous with $b \leq a$. 
We write $a \not \geq b$ to mean that $a \geq b$ is false, that is, 
{\em for some} $i \in \BFn$, $a_i < b_i$. 

We generalize the usual interval notation in ${\Bbb Z}$ to describe 
certain subsets of ${\Bbb Z}^n$:\\ 
for $a,b \in {\Bbb Z}^n,$ 
$$(-\infty, a] = \{c \in {\Bbb Z}^n \ : \ c \leq a\} 
\mbox{ and } [a,\infty) = \{c \in {\Bbb Z}^n  \ : \ a \leq c\}$$
$$[a,b] = \{ c \in {\Bbb Z}^n\ :\ a \leq c \leq b \} = \prod_{i=1}^n [a_i,b_i]$$
$$(a,b] = \{ c \in {\Bbb Z}^n\ :\ a \not \geq c \leq b\}
= (-\infty,b]\setminus (-\infty,a]$$
$$[a,b) = \{ c \in {\Bbb Z}^n\ :\ a \leq c \not \geq b\}
= [a,\infty) \setminus [b,\infty).$$

\subsection{Polynomials and Laurent series}
$R$ is a commutative ring with $1\neq 0$ and $P_n$ 
denotes the ring of $R$--polynomials in $X_1,\ldots,X_n$; 
$L_n$ denotes the ring of
Laurent series in $X_1^{-1},\ldots,X_n^{-1}$, which contains $P_n$.
Indeed, $L_n$ is a $P_n$--module which has $P_n$ as 
standard $P_n$--submodule, where the 
action of $X_i$ is {\em a shift to the right.} 
$S_n \subseteq L_n$ denotes the ring of $R$--power series in 
$X_1^{-1},\ldots,X_n^{-1}.$

For $a \in {\Bbb Z}^n$, we abbreviate $X_1^{a_1} X_2^{a_2} \ldots 
X_n^{a_n}$ to $\BFX^a$. For $i \in \BFn,\ \widehat{\bf X}_i$ denotes the
variables $X_1,\ldots,X_n$, excluding $X_i$.

For $G \in L_n\setminus\{0 \}, \ G_a$ denotes a coefficient of $G$, 
and $\mbox{Supp}(G) = \{ a \in {\Bbb Z}^n : G_a \neq 0 \}$ 
is called the {\em support} of $G; \ \mbox{Supp}(0) = \emptyset$. 
If $G \in L_n$ and $A \subseteq {\Bbb Z}^n$, then $G|A = \sum_{a \in A \cap
{\scriptsize Supp}(G)} G_a \BFX^a$.

We use the (exponential) valuation on $R((X^{-1}))$, which extends the 
degree function on $R[X]$. For convenience, we also denote it by $\delta;
\ \delta 0 = -\infty$. 
Thus for $G \in R((X^{-1}))\setminus\{0 \}$, 
$\delta G = \max \mbox{Supp}(G)$ and
for $G,H \in R((X^{- 1})), \delta (GH) \leq \delta G + \delta H$.

For $G \in L_n$ and $i \in \BFn$, we let $\delta _iG$ be the $i^{th}$ {\em 
partial degree} of $G$, that is, the degree of $G$ regarded as a Laurent 
series in $X_i^{-1}$; 
$\delta G$ is the $n$--vector with components $\delta _iG$. 
We have $\mbox{Supp}(G) \subseteq (-\infty,\delta G]$.

{\em The letter $f$ will always mean an element of $P_n$} and we 
write $f = f(\BFX) 
= \sum _{a \in {\scriptsize Supp}(f)} f_a \ \BFX^a.$

Supp$(f) \subseteq [\hbox{\bf 0},\delta f]$, and so for $a \in \mbox{Supp}(f), \ 
\delta f -a$ is a well-defined point in ${\Bbb N}^n.$
If $f \neq 0$, the reciprocal of 
$f$ is $f^* =\BFX^{\delta f}\ f(\BFX^{\, \hbox{\bf \scriptsize -1}}) = 
\sum _{a \in {\scriptsize Supp}(f)} f_a \ \BFX^{\delta f-a}$
and $0^*=0.$ It is easy to verify that 
$f = \BFX ^{\delta f - \delta f^*}f^{**}.$

\subsection{Linear recurring sequences} 

$S^n(R)^-$ denotes the set of functions $-{\Bbb N}^n 
\rightarrow R$ i.e. the set of $R$--sequences indexed by $-{\Bbb N}^n;$
the value $s(a)$ is written $s_a.$ {\em The letter $s$ will always 
denote a sequence in} $S^n(R)^-$, and to avoid trivial cases, we assume
that $s$ is non--zero. With addition and scalar product 
defined componentwise, $S^n(R)^-$ 
becomes a unitary $R$--module. Further, the unit sequence and the Hadamard 
product make $S^n(R)^-$ into a commutative $R$--algebra with $1$. In 
particular, its ideals are algebra ideals.
There is a {\em right} shift action of $P_n$ 
on $S^n(R)^-$ given by 
$$\left(\sum_{\hbox{\bf \scriptsize 0}\leq a\leq \delta f} 
f_a \BFX^a \circ s \right) _b 
= \sum_{\hbox{\bf \scriptsize 0}\leq a\leq \delta f} f_a s_{b-a}$$ 
where $b \leq \BFZ$. This makes $S^n(R)^-$ into a $P_n$--module. 

\begin{definition}  The {\em annihilator or characteristic ideal} of 
(characteristic polynomials) of $s$ is 
$\mbox{Ann}(s) = \{f \in P_n : f\circ s = 0 \}.$ We say that 
$s$ is a {\em (homogeneous) $n$--dimensional linear recurring 
sequence ($n$--D lrs)} if $\mbox{Ann}(s) \neq \{0 \}$.
\end{definition}

\begin{example} 

(a) Let $a \geq \BFZ$. Then $s_b = 0$ for $b \leq -a 
\Longleftrightarrow \BFX^a \in \mbox{Ann}(s)$. 

(b) Let $r \in R\setminus\{0\}, \ a \leq \BFZ$ and 
$s_a = r^{-\sum_{i=1}^n a_i}$. Then $\BFX - r^n \in \mbox{Ann}(s)$. 

(c) If $s \in S^2(R)^-$ then $s_{a-1,b} = s_{a,b} = s_{a,b-1} \mbox{ for }
a,b \leq \BFZ \Longleftrightarrow X_1-1, X_2-1 \in \mbox{Ann}(s)$. 

(d) We say that $s$ is $n$--{\em periodic} if 
for some $d \geq \BFone, 
s_{a-d_i} = s_{a}$ for all $a\leq \BFZ$ and $i \in \BFn$
i.e. $s$ is $n$--periodic iff for some $d \geq \BFone$,
$X_i^{d_i}-1 \in \mbox{Ann}(s)$ for all $i \in \BFn;\ 
s$ is {\em eventually $n$--periodic} if for some 
$d \geq \BFone, e \geq \BFZ, \ 
X_i^{e_i}(X_i^{d_i}-1) \in \mbox{Ann}(s)$ for all $i \in \BFn.$
\end{example}

We now generalize the notion of a rectilinear
$n$-D lrs over a field ${\Bbb F}$ from Fitzpatrick \& Norton (1990). 
(Recall that $s$ is rectilinear if for all $i \in \hbox{\bf n}$, there 
is an $f_i \in \mbox{Ann}(s) \cap {\Bbb F}[X_i]$ with $\delta f_i \geq 1$ 
and $f_i(0) \neq 0$.)

\begin{definition} We say that $s$ is {\em eventually rectilinear} (EVR) if 
for all $i \in \BFn$, there is an $f_i \in \mbox{Ann}(s) \cap R[X_i]$ with $\delta f_i \geq 
1$. 
\end{definition}

If $R$ is a domain and $s$ is non--zero, then $s$ is EVR 
$\Longleftrightarrow 
\mbox{Ann}(s) \cap R[X_i] \neq \{0 \}$ for all $i \in \BFn.$

EVR sequences include $1$--D lrs over a commutative ring, 
eventually periodic lrs (Nerode (1958)), 
rectilinear lrs over a field (Fitzpatrick \& Norton (1990)) and 
``$n$--linearly recursive functions over a field'' (Cerlienco 
\& Piras (1991)). If $s$ is EVR and $g \in P_n$, then $g\circ s$ is EVR since 
$f\circ (g\circ s) = (fg)\circ s = g\circ (f\circ s)$. 
 
\begin{example} \label{exampleab} 

(a) Let $s \in S^2(R)^-$ satisfy 
$s_{0,j} = 1$ for $j \leq 0$ and $s_{i,j} = 0$ otherwise. Then 
$X_1,\ X_2-1 \in \mbox{Ann}(s)$, so that $s$ is EVR. 
It is easy to check that $\mbox{Ann}(s) \cap R[X_1] = (X_1)$, 
so that $s$ is not rectilinear.

(b) Let $s' \in S^2(R)^-$ satisfy $s'_{a,b} = 1$ if $a = b$ 
and $s'_{a,b} = 0$ otherwise. Then $X_1X_2-1 \in \mbox{Ann}(s')$ 
but $\mbox{Ann}(s') \cap R[X_i] = \{0 \}$ for $i =1,\ 2$: 

if $f = \sum_{a=0}^{\delta f} f_a X_1^a \in R[X_1]$ and $f\circ s' = 0$ then 
for $0 \leq b \leq \delta f$,
$$f_b = \sum_{a=0}^{\delta f} f_a s'_{-a,-b} = 
\left(\sum_{a=0}^{\delta f} f_a X_1^a \circ s'\right)_{0,-b} 
= 0.$$
Similarly, $\mbox{Ann}(s') \cap R[X_2] = \{0 \}$. Thus $s'$ is {\em not} EVR. 
In particular, a 2--D lrs over GF(2) need not be 2--periodic (in contrast to
the 1--D situation).
\end{example}

In the literature, the standard approach to $n$--D lrs 
is to regard  $S^n(R) = \{ s : {\Bbb N}^n \rightarrow R \}$ as a $P_n$--module 
via {\em left} shifting (which we also denote by $\circ$ ):
 
$$\left(\sum_{\hbox{\bf \scriptsize 0}\leq a\leq \delta f} 
f_a \BFX^a \circ s \right) _b 
= \sum_{\hbox{\bf \scriptsize 0}\leq a\leq \delta f} f_a s_{b+a}$$ 
where $b\geq\BFZ$. 

\begin{proposition} \label{minus} The map $^-: S^n(R) \rightarrow S^n(R)^-$ 
given by $(s^-)_a = s_{-a}$ for $a \leq \BFZ$ is a $P_n$--module map.
\end{proposition}

 \bpr Let $s \in S^n(R), \ b \leq \BFZ$ and 
$f = \sum_{ \hbox{\bf \scriptsize 0}\leq a \leq \delta f} f_a \BFX^a 
\in P_n$. Then 
$$(f\circ s)^-_b = (f\circ s)_{-b} 
= \sum_{\hbox{\bf \scriptsize 0}\leq a\leq \delta f} f_a s_{-b+a} 
= \sum_{\hbox{\bf \scriptsize 0}\leq a\leq \delta f} f_a s^-_{b-a} 
= (f\circ s^-)_b$$ 
so that $(f\circ s)^- = f\circ s^-$. 
It is trivial that $^-$ preserves addition, so $^-$ is a 
$P_n$--module map.
\epr

There is an obvious two--sided inverse of $^-$, so that $S^n(R)^-$ and 
$S^n(R)$ are isomorphic $P_n$--modules,
$s$ is an lrs iff $s^-$ is an lrs, and $^-$ 
induces an isomorphism of $\mbox{Ann}(s)$ and $\mbox{Ann}(s^-)$. 

\subsection{Ann(s) and generating functions}

The {\em generating function} of $s$ is 
$$\Gamma (s) = \sum_{a\leq \hbox{\bf \scriptsize 0}} s_a \BFX^a 
\in S_n.$$

Strictly speaking, the definition of $\Gamma(s)$ should incorporate a
linear (total) ordering on $-{\Bbb N}^n$; this has been omitted since the 
enumeration will either be clear or not used.
Thus for Example~\ref{exampleab}, $\Gamma(s) = X_2/(X_2-1)$ and 
$\Gamma(s') = X_1X_2/(X_1X_2-1).$

Since $\delta (f\Gamma (s)) \leq \delta f$, we have
$\mbox{Supp}(f\Gamma (s))\subseteq (-\infty,\delta f]$.

\begin{proposition} \label{porism} 
$ (f\Gamma (s))|(-\infty,\hbox{\bf 0}] = \Gamma(f\circ s).$
\end{proposition}

\begin{proof} If $a \leq \BFZ$ and 
$f = \sum_{\hbox{\bf \scriptsize 0}\leq b\leq \delta f} f_b \BFX^b$, then 
$(f\Gamma (s))_a = 
\sum_{ \hbox{\bf \scriptsize 0}\leq b\leq \delta f} f_bs_{a-b} 
= (f\circ s)_a.$ 
\end{proof}

We note that Lemma 4.1 of Fitzpatrick \& Norton (1990) is an analogue
of the preceding result for left--shifting.

\begin{proposition} \label{charpoly} The following are equivalent:

 (a) $f \in \mbox{Ann}(s)$ 

 (b) $(f\Gamma (s))|(-\infty,\BFZ] = 0$ 

 (c) $\mbox{Supp}(f\Gamma (s)) \subseteq (\BFZ, \delta f]$

 (d) $\exists G \in L_n$ such that $f\Gamma (s) = \BFX G$ and 
$\mbox{Supp}(G) \subseteq (-\BFone, \delta f - \BFone]$.
\end{proposition}

\begin{proof}
We omit the proof that (a) $\Longleftrightarrow$ 
(b) $\Longleftrightarrow$ (c).
(c) $\Longleftrightarrow$ (d): It is a straightforward exercise that 
$\mbox{Supp}(f\Gamma (s)) \subseteq (\BFZ, \delta f] \ 
\Longleftrightarrow 
\mbox{Supp}(\BFX^{-\hbox{\bf \scriptsize 1}} 
f \Gamma (s) )\subseteq (-\BFone, \delta f-\BFone]$.
\end{proof}

Setting $n=1$ in Proposition~\ref{charpoly}, 
we obtain a simple characterization of $\mbox{Ann}(s)$  which 
does not use use $\delta f$ or the reciprocal of $f$ 
({\em cf.} Niederreiter (1988), Lemmas 1, 2 when $R$ is a field):

\begin{corollary}\label{1dcharpoly} Let $s \in S^1(R)^-$. 
Then $f \in \mbox{Ann}(s) 
\Longleftrightarrow 
f\Gamma (s) \in XR[X].$
\end{corollary}

\begin{corollary} \label{ndlrs} Let $s$ be an lrs. Then 
$\Gamma (s) = \BFX G/f$ for 
some non--zero $f \in \mbox{Ann}(s)$ and $G \in L_n$ such 
that $\mbox{Supp}(G) \subseteq (-\BFone, \delta f- \BFone]$. 

Conversely, if $f$ is monic, $G \in L_n$ 
and $\mbox{Supp}(G) \subseteq (-\BFone,\delta f-\BFone]$, then 
$\BFX G/f \in S_n$ and
$s$ defined by $\Gamma (s) = \BFX G/f$ is an lrs with $f \in \mbox{Ann}(s)$.
\end{corollary}

\begin{proof} If $f$ is monic, then 
$u = \BFX^{-\delta f}f$ is a unit in 
$S_n$. Writing $G = 
\sum_{a \in (-\hbox{\bf \scriptsize 1},\delta f-\hbox{\bf \scriptsize 1}]} 
G_a \BFX^a$, we obtain 
$$\BFX G u = f\
\sum_{ a \in (-\hbox{\bf \scriptsize 1}, \delta f-\hbox{\bf \scriptsize 1}]} 
G_a \ \BFX^{a+\hbox{\bf \scriptsize 1}-\delta f} 
= f  \sum_{ b \in (-\delta f,  \hbox{\bf \scriptsize 0}]} 
G_{b-\hbox{\bf \scriptsize 1} +\delta f} \ \BFX^b = f H$$ 
say, where $H \in S_n$ and so $\BFX G/f 
= H/u \in S_n.$ The result now follows from Proposition~\ref{charpoly}.
\end{proof}

``Rational lrs'' are an important special case of Corollary~\ref{ndlrs}:

\begin{definition} We will say that $s$ is a {\em rational} lrs if for some
$f \in P_n \setminus\{0\},\ f\cdot\Gamma(s) \in \BFX P_n.$ 
\end{definition}

Certainly every 1--D lrs is rational, but this fails for $n \geq 2$: 
if $G(X_1^{-1}) \in S_1$ is not a rational function, 
$f \in P_n,\ \delta f \geq \BFone$, and $s$ is defined by 
$\Gamma(s) = \widehat{\bf X}_1G/f$, then $f \in \mbox{Ann}(s)$,
but $s$ is not rational. Example~\ref{exampleab}(b) shows that for $n \geq 2$,
not all rational lrs are EVR. We will see in Theorem~\ref{evrgamma} 
below that EVR lrs are always rational. 

For completeness, we now give an analogue of Corollary~\ref{ndlrs}
for sequences indexed by ${\Bbb N}^n$ which use 
left--shifting (Corollary~\ref{leftgamma} below). 
This analogue could have been deduced from 
Propositions~\ref{minus}, ~\ref{charpoly} 
and an analogue of Proposition~\ref{charpoly}, 
but it is easier to appeal directly to the definitions:

\begin{lemma} \label{leftshift} Let $a \geq \hbox{\bf 0}$ and 
$s : {\Bbb N}^n \rightarrow R$ (so that $\Gamma(s) \in R[[\BFX]]$). Then 
$$\BFX^a f^\ast \in \mbox{Ann}(s) \Longleftrightarrow 
\mbox{Supp}(f\G(s)) \subseteq [\hbox{\bf 0},a+\delta f).$$
In particular, if $\mbox{Supp}(f\G(s)) 
\subseteq [\hbox{\bf 0},\delta f)$, then $f^\ast \in \mbox{Ann}(s)$.
\end{lemma}

\begin{proof} For any $b \geq \hbox{\bf 0}$, 
\begin{eqnarray*}
\left(\BFX ^af^\ast\circ s \right)_b 
& = & \left( \sum_{c\in Supp(f)}
\ f_c \BFX^{a+\delta f-c} \circ \ s \right)_b\\
& = & \sum_{c\in Supp(f)} f_c s_{b+a+\delta f-c} 
= ( f\G(s) )_{a+b+\delta f}\\
\end{eqnarray*}
and so 
$$\BFX^a f^\ast \in \mbox{Ann}(s) 
\Longleftrightarrow  (f\G(s))|[a+\delta f,\infty) = 0
\Longleftrightarrow \mbox{Supp}( f\G(s) ) \subseteq [\BFZ,a+\delta f).$$
\end{proof}

\begin{corollary} \label{leftgamma}
Let $s : {\Bbb N}^n \rightarrow R$ be an lrs. Then
$\Gamma(s) = G/f^*$ for some non--zero $f \in Ann(s)$ and $G \in R[[\BFX]]$ 
such that Supp$(G) \subseteq[\hbox{\bf 0}, \delta f).$

Conversely, suppose that $f(\hbox{\bf 0})=1$ and $G \in 
R[[\BFX]]$ satisfies $\ \mbox{Supp}(G) \subseteq [\hbox{\bf 0},d)$ 
for some $d \in {\Bbb N}^{\hbox{\bf \scriptsize n}} \setminus 
\{\hbox{\bf 0} \}$. Define $m \in {\Bbb N}^n$ by 
$$\pi_i m = \max\{0, d_i - \delta _if \}.$$ 
If $s$ is the $n$--D lrs with $\G(s) = G / f$ and $a \geq m$, then 
$\BFX^af^\ast \in \mbox{Ann}(s)$.
\end{corollary}

\begin{proof} The first part is an easy consequence of Lemma~\ref{leftshift}
applied to $f = \BFX^{\delta f - \delta f^\ast}(f^\ast)^\ast$.

For the converse, if $a \geq m$, then $a \geq \hbox{\bf 0}$ and 
$d \leq a + \delta f$. 
Therefore $[\hbox{\bf 0},d) \subseteq [\hbox{\bf 0},a+\delta f)$, 
so the result follows from Lemma~\ref{leftshift}.
\end{proof}

Here, rational lrs correspond to $\G(s) = g / f^\ast$, 
where $f, g \in P_n,\ \delta f \geq \BFone$ and  
$\delta g \leq \delta f - \BFone$; 
in this case, we can take $d = \delta g+\hbox{\bf 1}$ in 
Corollary~\ref{leftgamma}. 
If in addition $n = 1$ in Corollary~\ref{leftgamma}, we obtain 
Fitzpatrick \& Norton (1995), Theorem 4.1.

\section{A decomposition formula for $f \cdot \Gamma(s)$}

\subsection{The border partition} 

For $d \geq \BFone, \ (-\infty,d] = \prod_{i=1}^n(-\infty,d_i]$ and so 
$$\beta(-\infty,d] = \prod_{i=1}^n \{(-\infty,0],[1,d_i]\}$$
is a partition of $(-\infty,d]$.  
We will abbreviate $\beta(-\infty,d]$ to $\beta$ when $d$ is understood.

It is convenient to use binary notation to index the $2^n$ members of
the product partition 
$\beta(-\infty,d]$: \ for $0 \leq k \leq 2^n-1$, we write the ($n$--bit)
binary representation of $k$ with the most significant bit $b_n(k)$ of $k$
on the right (so that the bit ordering is the same as the ordering of the
components of ${\Bbb Z}^n$) 
and index the $k^{th}$ member $\beta_k(-\infty,d]$ of $\beta(-\infty,d]$ by
$$b_i(k) = 1 \mbox{ iff } \pi_i(\beta_k(-\infty,d]) = (-\infty,0].$$

For example, if $n = 2,\ 
\beta_0 = \beta_{00} = [1,d_1] \times [1,d_2],\ 
\beta_1  = \beta_{10} = [1,d_1] \times (-\infty,d_2], \
\beta_2 = \beta_{01} = (-\infty,d_1] \times [1,d_2]$ 
etc.; notice that $\beta_0 \cup \beta_1 \cup \beta_2$ "borders"
$\beta_3 = (-\infty,\BFZ]$ in $(-\infty,d].$

We will call $\bigcup \{\beta_k(-\infty,d] \ : 0 \leq k < 2^n-1\}$ 
the {\em border}
of $(-\infty,\BFZ]$ in $(-\infty,d]$ and $\beta(-\infty,d]$ the {\em
border partition} of $(-\infty,d]$.

\subsection{The border polynomial}
\begin{definition} 
Let $\delta f \geq \BFone$ and let $\beta$ be the border
partition of $(-\infty,\delta f].$ 
For $0 \leq k \leq 2^n-1$, 
we define $\beta_k(f,s) \in L_n$ by
$$\beta_k(f,s) = (f\cdot\Gamma(s) )\ |\ \beta_k.$$ 

\end{definition}

\begin{proposition} \label{simpledecomp}
Let $\delta f \geq \BFone$ and let $\beta$ be the border
partition of $(-\infty,\delta f].$  Then
$$f \cdot \Gamma(s) = \sum_{k=0}^{2^n-1}\beta_k(f,s)$$
where $Supp(\beta_k(f,s)) \subseteq \beta_k$ for $0 \leq k \leq 2^n-1$ and
$\beta_{2^n-1}(f,s) = \Gamma(f \circ s).$
\end{proposition}

\begin{proof} This follows immediately from the fact that 
$\beta$ is a partition of $(-\infty,\delta f]$ and from
Proposition~\ref{porism}.
\end{proof}

Since $\beta_0(f,s) \in P_n$, we call it the {\em border polynomial} of
$f$ and $s$. We will express $\beta_0(f,s)$ in terms of $s$ and 
generalized Newton divided differences of $f$. 
Recall Newton's divided difference operator $\nu$ for $n=1:$\\ 

$\nu: R[X] \rightarrow R[X]$ is $\nu f = (f -f_0)/X$, and for $a \geq 1,\
\nu^a$ is defined iteratively by $\nu^0 f = f \mbox{ and } 
\nu^a = \nu \nu^{a-1}.$
It is easy to see that if $d \geq 1$ 
and $f = f_0 + f_1X +\ldots + f_dX^d$ , then for $0 \leq a \leq d$, 
$$\nu^a f = f_a\ +\ f_{a+1}\ X + \ldots +\ f_d X^{d-a} = \sum_{b=a}^d f_b X^{b-a}$$
whereas $\nu^a f = 0$ if $a > \delta f.$
Thus $\delta (\nu^af) = \delta f - a$ if $0 \leq a \leq \delta f$.

\begin{definition} For $n \geq 1, \ a \geq \hbox{\bf 0}$ and $f \in P_n$, 
we define
$$\nu^a f 
= \left( f \ | \ [a,\delta f]\right) / \BFX^a.$$
\end{definition}

Clearly $\nu^{\hbox{\bf \scriptsize 0}} f = f$ and the 
numerator is divisible by $\BFX^a$. 
Indeed, if $\hbox{\bf 0} \leq a \leq \delta f$, then 
$\nu^a f = \sum_{a\leq b\leq \delta f}f_b \BFX^{b-a}$ and 
$\delta (\nu^a f) \leq \delta f-a$. 
For example, if $\delta f = (2,3)$ then 
$$\nu^{(1,1)} f = f_{1,1} + f_{1,2}X_2 + f_{2,1}X_1 + 
f_{1,3}X_2^2 + f_{2,2}X_1X_2 + f_{2,3}X_1X_2^2.$$ 

On the other hand, if $a \geq \delta f+\hbox{\bf 1}$, then $\nu^a f = 0.$    
When $n = 1,\ \nu^a$ coincides with Newton's operator.

\begin{lemma}\label{shortproof} If $\delta f \geq \BFone$, then 
$$\beta_0(f,s)  
= \BFX 
\sum_{\hbox{\bf \scriptsize 0}\leq a\leq \delta f-\hbox{\bf \scriptsize 1}} 
s_{-a}\  \nu ^{a+\hbox{\bf \scriptsize 1}} f.$$ 

\end{lemma}

\begin{proof} 
\begin{eqnarray*}
\beta_0(f,s) 
& = & \sum_{\hbox{\bf \scriptsize 1}\leq a\leq \delta f}
\left(
f\Gamma (s)
\right)_a \BFX^a \\
& = & \sum_{\hbox{\bf \scriptsize 1}\leq a\leq \delta f}
\left(
\sum_{a\leq b\leq \delta f} f_bs_{a-b}
\right) 
\BFX^a\\
& = & \sum_{\hbox{\bf \scriptsize 1}\leq a\leq \delta f}
\left( 
\sum_{\hbox{\bf \scriptsize 0}\leq c\leq \delta f-a}
f_{a+c}s_{-c} \BFX^a 
\right) \\
& = & \sum_{\hbox{\bf \scriptsize 0}\leq c\leq \delta f-\hbox{\bf \scriptsize 1}}
s_{-c}
\left( 
\sum_{\hbox{\bf \scriptsize 1}\leq a\leq \delta f-c}f_{a+c} \BFX^a
\right)\\
& = & \sum_{\hbox{\bf \scriptsize 0}\leq c\leq \delta f-\hbox{\bf \scriptsize 1}}
s_{-c}
\left(
\sum_{c+\hbox{\bf \scriptsize 1}\leq d\leq \delta f}f_d
\BFX^{d-c}
\right)\\
& = & \BFX \sum_{\hbox{\bf \scriptsize 0}\leq c\leq \delta f-\hbox{\bf \scriptsize 1}} 
s_{-c}\left(
\sum_{c+\hbox{\bf \scriptsize 1}\leq d\leq \delta f}f_d \BFX^{d-(c+\hbox{\bf 
\scriptsize 1})}
\right)\\
& = & \BFX \sum_{\hbox{\bf \scriptsize 0}\leq c\leq \delta f-\hbox{\bf 
\scriptsize 1}} s_{-c} \ 
\nu ^{c+\hbox{\bf \scriptsize 1}} f.\\
\end{eqnarray*}
\end{proof}

Setting $n=1$, we obtain the following result ({\em cf.} Ferrand (1988),
Section 1.3):

\begin{corollary} If $n=1$ and $\delta f \geq 1$, then 
$$f\Gamma(s) = X \sum_{a=0}^{\delta f-1} s_{-a} \ \nu ^{a+1} f \ 
+ \Gamma(f \circ s).$$
\end{corollary}


\subsection{The Border Laurent Series}

We have determined the summands $\beta_k(f,s)$ of $f\cdot\Gamma(s)$ in 
Proposition~\ref{simpledecomp} for $k=0,\ 2^n-1.$ 
In this subsection, we determine the remaining summands (when $n \geq 2$).

If $0 < k < 2^n-1,$ then
for some $i,j , \ 1 \leq i,j \leq n,\ b_i(k) =0 $ and $b_j(k) = 1$, so
that for these $i,j, \ \beta_k(f,s)$
involves polynomials in $X_i$ and power series in $X_j^{-1}$. 
For this reason, we call $\{\beta_k(f,s) \ : \ 0 < k < 2^n-1 \}$ 
the {\em border Laurent series of $f$ and $s$}.

In order to discuss these series, we need to extend some of our earlier 
definitions:

$\BFi$ denotes a subset of  $\BFn = \{1,2,\ldots ,n \}$ with cardinality
$|\BFi|$ and 
$\BFi ' = \BFn\setminus \BFi;\
\BFi \ll \BFn$ means that $\BFi$ is a {\em proper} subset of $\BFn.$ 
If $\BFi \ll \BFn$ and $a \in {\Bbb N}^{\hbox{\bf \scriptsize i}}, \
X_{\hbox{\bf \scriptsize i}}^a$ has the obvious meaning 
and for $1 \leq i \leq n$, we write $X_i$ for $X_{\{i \}}$. 
We let $\pi_{\hbox{\bf \scriptsize i}}$ be the projection of ${\Bbb Z}^n$ 
onto ${\Bbb Z}^{\hbox{\bf \scriptsize i}}$. 
If $a\in {\Bbb Z}^{\hbox{\bf \scriptsize i}}$ and 
$a' \in {\Bbb Z}^{\hbox{\bf \scriptsize i}'}$, 
then $(a,a')$ is the unique element of 
${\Bbb Z}^n$ with $\pi_{\hbox{\bf \scriptsize i}}(a,a') = a$ 
and $\pi_{\hbox{\bf \scriptsize i}'}(a,a') = a'$. Lastly, 
for $\delta f \geq \BFone,\ 
\delta_{\hbox {\bf \scriptsize i}}f = \pi_{\sBFi} \delta f.$


\begin{definition} Let $n \geq 2$ and $\hbox{\bf i} \ll \hbox{\bf n}$. 

(a) For $a \in -{\Bbb N}^{\hbox{\bf \scriptsize i}}$, the $a$--{\em section} 
of $s$ is $s^{(a)} \in S^{|\hbox{\bf \scriptsize i}'|}(R)$ given by 
$$s^{(a)}_{\ a'} = s_{(a,a')}$$
for $a' \in -{\Bbb N}^{\hbox{\bf \scriptsize i}'}$.

(b) For $\hbox{\bf 0} \leq a \leq \delta _{\hbox{\bf \scriptsize i}}f$, the
$a$--{\em section} of $f$ is $f^{(a)}\in 
R[X_{\hbox{\bf \scriptsize i}'}]$ given by
$$f^{(a)}(X_{\hbox{\bf \scriptsize i}'}) = 
\sum_{\hbox{\bf \scriptsize 0}'\leq a'\leq \delta_{\hbox{\bf \scriptsize i}'}f}
f_{(a,a')} X_{\hbox{\bf \scriptsize i}'}^{a'}.$$
\end{definition}

Clearly $\delta f^{(a)} \leq \delta_{\hbox{\bf \scriptsize i}'}f$ and
for $\BFZ' \leq a' \leq \delta (f^{(a)}),\ f^{(a)}_{\ a'} = f_{(a,a')}.$


We will simplify the main 
result of this subsection by using the following technical definition:

\begin{definition} Let $n \geq 2$ and $\delta f \geq \hbox{\bf 1}.$ 
For $\BFi \ll \BFn$, let $\delta = \delta_{\sBFi}f$ and 
$\delta' = \delta_{\sBFi'}f.$
We will write the sum of ``cross--products of sections'' 
$$X_{\hbox{\bf \scriptsize i}'} 
\sum_{\hbox{\bf \scriptsize 0}\leq a\leq \delta}\ 
\sum_{\hbox{\bf \scriptsize 0}'\leq a'\leq \delta'-\hbox{\bf \scriptsize 1}'} 
\left( \nu^{a'+\hbox{\bf \scriptsize 1}'} f^{(a)}\right) 
(X_{\hbox{\bf \scriptsize i}'})  
\cdot
\G\left(X_{\hbox{\bf \scriptsize i}}^a\circ s^{(-a')}\right)
(X_{\hbox{\bf \scriptsize i}}^{-\hbox{\bf \scriptsize 1}})$$
as $\beta^{\times}_{\sBFi}(f,s)
(X_{\hbox{\bf \scriptsize i}}^{-\hbox{\bf \scriptsize 1}},
X_{\hbox{\bf \scriptsize i}'})$
for short.

\end{definition}

Notice that 
$\beta^{\times}_{\sBFi}(f,s)
(X_{\hbox{\bf \scriptsize i}}^{-\hbox{\bf \scriptsize 1}},
X_{\hbox{\bf\scriptsize i}'}) \in L_n$. 
In fact it belongs to $\left( R[X_{\hbox{\bf \scriptsize i}'}]\right)
[[X_{\hbox{\bf \scriptsize i}}^{-\hbox{\bf \scriptsize 1}}]]$.
We will also abbreviate 
$\beta^{\times}_{\sBFi}(f,s)
(X_{\hbox{\bf \scriptsize i}}^{-\hbox{\bf \scriptsize 1}},
X_{\hbox{\bf\scriptsize i}'})$
to 
$\beta^{\times}_{\sBFi}(f,s).$

For $n = 2$ and continuing the notation of the definition, 
$\BFi = \{1\} \mbox{ or } \ \{2\}$, and
$$\beta^{\times}_{\{1\}}(f,s)(X_1^{-1} , X_2) = X_2 
\sum_{a=0}^{\delta}
\sum_{a'=0}^{\delta'-1}
\left(\nu^{a'+1} f^{(a)}\right)(X_2) 
\cdot
\G\left(X_1^a\circ s^{(-a')}\right)(X_1^{-1})$$
$$\beta^{\times}_{\{2\}}(f,s)(X_2^{-1}, X_1)
= X_1 \sum_{a=0}^{\delta}
\sum_{a'=0}^{\delta'-1} 
\left( \nu^{a'+1} f^{(a)}\right)(X_1)
\cdot
\G\left(X_2^a\circ s^{(-a')}\right)(X_2^{-1})$$
where $\nu$ denotes the usual divided difference operator.

For $0 < k < 2^n-1$, we define $\BFi(k) = \{ i \in \BFn \ : \ b_i(k) = 1 \}
\ll \BFn$.
Thus if $\delta f \geq \hbox{\bf 1},\ \BFi(k)$ is the set of $i$ for which 
$\pi_i\beta_k(-\infty,\delta f]$ is infinite.

\begin{theorem} \label{borderseries} Let $n \geq 2, \ 0 < k < 2^n-1$, and let
$\BFi = \BFi(k)$. Then for $\delta f \geq \hbox{\bf 1}$,
the $k^{th}$ border Laurent series of $f$ and $s$ is 
$$\beta_k(f,s) = 
\beta^{\times}_{\sBFi}(X_{\hbox{\bf \scriptsize i}}^{\hbox{\bf \scriptsize -1}},
X_{\hbox{\bf \scriptsize i}'}).$$
\end{theorem}

\begin{proof} 
\begin{eqnarray*}
\beta_k(f,s)
& = & 
\sum_{(a,a')\in\beta_k(-\infty,\delta f]\cap\mbox{\scriptsize Supp}(f\Gamma)}
\left(f\Gamma\right)_{(a,a')} \BFX^{(a,a')}\\
& = & \sum_{a \leq \hbox{\bf \scriptsize 0}, \
\hbox{\bf \scriptsize 1}' \leq a' \leq \delta'}
\left( \sum_{(a,a')\leq b\leq (\delta,\delta')} 
f_b\ s_{(a,a')-b}\right) \BFX^{(a,a')}\\
& = & \sum_{a \leq \hbox{\bf \scriptsize 0},  \
\hbox{\bf \scriptsize 1}' \leq a' \leq \delta'}
\left( \sum_
{(\hbox{\bf \scriptsize 0}, \hbox{\bf \scriptsize 0}')
\leq (c,c')\leq (\delta,\delta')-(a,a')} 
f_{(a,a')+(c,c')}\ s_{-(c,c')}\right) \BFX^{(a,a')}\\
& = & \sum_{a \leq \hbox{\bf \scriptsize 0},   \
\hbox{\bf \scriptsize 0}' \leq c' \leq \delta' - \hbox{\bf \scriptsize 1}' }
\left( \sum_{\hbox{\bf \scriptsize 0}\leq c \leq \delta-a,\
\hbox{\bf \scriptsize 1}' \leq a' \leq \delta'} 
f_{(a,a')+(c,c')}\ s_{-(c,c')}\right) X_{\hbox{\bf \scriptsize i}}^a 
X_{\hbox{\bf \scriptsize i}'}^{a'}\\
& = & \sum_{d \leq \hbox{\bf \scriptsize 0}, \
\hbox{\bf \scriptsize 0}' \leq c' \leq \delta' - \hbox{\bf \scriptsize 1}' }
\left( \sum_{\hbox{\bf \scriptsize 0}\leq c \leq \delta-d,\
c' + \hbox{\bf \scriptsize 1}' \leq d' \leq \delta'} 
f_{(c+d,d')}\ s_{-(c,c')}\right) X_{\hbox{\bf \scriptsize i}}^d 
X_{\hbox{\bf \scriptsize i}'}^{d'-c'}\\
& = & X_{\hbox{\bf \scriptsize i}'}\sum_{d \leq \hbox{\bf \scriptsize 0}, \
\hbox{\bf \scriptsize 0}' \leq e' \leq \delta' - \hbox{\bf \scriptsize 1}' }
\left( \sum_{d\leq e \leq \delta,\
e' + \hbox{\bf \scriptsize 1}' \leq d' \leq \delta'} 
f_{(e,d')} X_{\hbox{\bf \scriptsize i}'}^{d'}\ s_{(d-e,e')}\right)
X_{\hbox{\bf \scriptsize i}}^d\\
& = & X_{\hbox{\bf \scriptsize i}'}
\sum_{ \hbox{\bf \scriptsize 0}\leq e \leq \delta, \
\hbox{\bf \scriptsize 0}' \leq e' \leq \delta' - \hbox{\bf \scriptsize 1}' }
\left(
\sum_{d\leq \hbox{{\bf \scriptsize 0}},\
e' + \hbox{\bf \scriptsize 1}' \leq d' \leq \delta'} 
f_{(e,d')} X_{\hbox{\bf \scriptsize i}'}^{d'-(e'+\hbox{\bf \scriptsize 1}')} 
\ s_{(d-e,e')}
X_{\hbox{\bf \scriptsize i}}^{d}
\right)\\
& = & X_{\hbox{\bf \scriptsize i}'}
\sum_{ \hbox{\bf \scriptsize 0}\leq e \leq \delta, \
\hbox{\bf \scriptsize 0}' \leq e' \leq \delta' - \hbox{\bf \scriptsize 1}' }
\left(
\sum_{
e' + \hbox{\bf \scriptsize 1}' \leq d' \leq \delta'} 
f^{(e)}_{d'} X_{\hbox{\bf \scriptsize i}'}^{d'-(e'+\hbox{\bf \scriptsize 1}')}
\right)
\cdot
\left(
\sum_{d \leq \hbox{\bf \scriptsize 0}}
\left( X_{\hbox{\bf \scriptsize i}}^e \circ s^{(-e')}\right)_d
X_{\hbox{\bf \scriptsize i}}^d
\right)\\
& = & X_{\hbox{\bf \scriptsize i}'}
\sum_{ {\hbox{\bf \scriptsize 0}} \leq e \leq \delta, \
\hbox{\bf \scriptsize 0}' \leq e' \leq \delta' - \hbox{\bf \scriptsize 1}' }
\left(\nu^{e' + \hbox{\bf \scriptsize 1}'} f^{(e)}\right) 
(X_{\hbox{\bf \scriptsize i}'})
\cdot
\Gamma \left(X_{\hbox{\bf \scriptsize i}}^e \circ s^{(-e')}\right)
(X_{\hbox{\bf \scriptsize i}}^{\hbox{\bf \scriptsize -1}})\\
& = & \beta^{\times}_{\sBFi}
(X_{\hbox{\bf \scriptsize i}}^{\hbox{\bf \scriptsize -1}},
X_{\hbox{\bf \scriptsize i}'}).\\
\end{eqnarray*}

\end{proof}

Let $\delta f \geq \BFone.$
To simplify the notation, we let 
$$\beta^\times_{\hbox{\bf \scriptsize i}(0)}(f,s) =
\BFX 
\sum_{\hbox{\bf \scriptsize 0}\leq a\leq \delta f-\hbox{\bf \scriptsize 1}} 
\ s_{-a} \ \nu ^{a+\hbox{\bf \scriptsize 1}} f,$$
which is just $\beta_0(f,s)$ by Lemma~\ref{shortproof}.

\begin{corollary} Let $\delta f \geq \hbox{\bf 1}.$

(a) 
$$f 
\cdot \Gamma(s) 
= \sum_{k=0}^{2^n-2}
\beta^{\times}_{\sBFi(k)}(f,s)
+
\Gamma(f\circ s)$$
where Supp$(\beta^{\times}_{\sBFi(k)}(f,s)) \subseteq \beta_k(-\infty,\delta f]$
for $0 \leq k \leq 2^n-2.$

(b) $f \in Ann(s)$ if, and only if
$$\Gamma(s) 
= \left( \sum_{k=0}^{2^n-2} \beta^{\times}_{\sBFi(k)}(f,s) \right) / f.$$

(c) If $f$ is monic, $G \in L_n$ 
and $\mbox{Supp}(G) \subseteq (-\BFone,\delta f-\BFone]$, then 
$\BFX G/f \in S_n$ and
$s$ defined by $\Gamma (s) = \BFX G/f$ is an lrs with $f \in \mbox{Ann}(s)$,
and $\BFX G = \sum_{k=0}^{2^n-2}
\beta^{\times}_{\sBFi(k)}(f,s).$
\end{corollary}

\begin{proof}
These are immediate consequences of Proposition~\ref{porism}, 
Corollary~\ref{ndlrs}, Lemma~\ref{shortproof} and Theorem~\ref{borderseries}.
\end{proof}

We now state an analogue of the previous corollary for sequences ${\Bbb N}^n
\rightarrow R$ and left--shifting. In this case, 
$\prod_{i=1}^n\{[-\delta_i f,-1], [0, \infty)\}$
is to be used as the ``border partition'' $\{ \beta_k[-\delta f,\infty)
: 0 \leq k \leq 2^n-1\}$ of $[-\delta f, \infty).$

\begin{corollary} \label{leftdecomp}
Let $\delta f \geq \hbox{\bf 1}$. If
$s :{\Bbb N}^n \rightarrow R$ (so that $\Gamma(s) \in R[[\BFX]]$), then

(a) 
$$f^\ast \G(s) =
\sum_{k=0}^{2^n-2}  
\beta^\star_{\hbox{\bf \scriptsize i}(k)}(f,s) + 
\BFX^{\delta f}\Gamma(f \circ s)$$ 
where 
$$\beta^\star_{\hbox{\bf \scriptsize i}(0)}(f,s) = 
\BFX^{\delta f}\ 
\beta^\times_{\hbox{\bf \scriptsize i}(0)}(f,s)
(\BFX^{-\hbox{\bf \scriptsize 1}}) \mbox{ and }
\beta^\star_{\hbox{\bf \scriptsize i}(k)}(f,s) = \BFX^{\delta f}
\beta^\times_{\hbox{\bf \scriptsize i}(k)}(f,s)
(X_{\hbox{\bf \scriptsize i}},
X_{\hbox{\bf \scriptsize i}'}^{-\hbox{\bf \scriptsize 1}})$$ for 
$0 \leq k \leq 2^n-2$ and
Supp$(\beta^\star_{\hbox{\bf \scriptsize i}(k)}) \subseteq
\{a+\delta f \ : \ a \in \beta_k[-\delta f, \infty)\}$ 
for $0 \leq k \leq 2^n-2.$

(b) If $f \in 
\mbox{Ann}(s)$, then 
$$\G(s) =
\left( 
\sum_{k=0}^{2^n-2}  
\beta_{\hbox{\bf \scriptsize i}(k)}^\star(f,s)
\right) / f^\ast$$ 
where $\beta_{\hbox{\bf \scriptsize i}(0)}^\star(f,s) \in P_n.$ In fact, 
$\delta \beta_{\hbox{\bf \scriptsize i}(0)}^\star(f,s) 
\leq \delta f - \BFone$ and
$\beta_{\hbox{\bf \scriptsize i}(0)}^\star(f,s)=
\BFX^{\delta f - \delta\beta_0(f,s)}\beta_0(f,s)^\ast.$

(c) If $f$ is monic and $G \in 
R[[\BFX]]$ satisfies $\mbox{Supp}(G) \subseteq [\hbox{\bf 0},\delta f)$, 
then $s$ defined by $G/f^\ast$ is an lrs with $f \in \mbox{Ann}(s)$ and 
$$G = \sum_{k=0}^{2^n-2} \beta^\star_{\hbox{\bf \scriptsize i}(k)}(f,s).$$

\end{corollary}
\begin{proof}
These are straightforward consequences of the previous result.
\end{proof}
\section{EVR sequences}

\subsection{1--D sequences}

It is convenient to begin with some results on 1--D sequences.

\begin{proposition} \label{1dgamma} Let $s \in S^1(R)^-$ 
and $\delta f \geq 1$.

(a) $f\Gamma(s) = \beta_0(f,s) + \Gamma(f\circ s)$

(b) $f \in \mbox{Ann}(s) \Longleftrightarrow f \Gamma (s) = \beta _0(f,s)$

(c) If $f \in \mbox{Ann}(s)$, then 
$\mbox{Ann}(s) = (f : \beta _0(f,s)/X).$
\end{proposition}

\begin{proof} (a): This is immediate from Proposition~\ref{simpledecomp}.

(b): $f \in \mbox{Ann}(s) \Longleftrightarrow f\circ s = 0 
\Longleftrightarrow \Gamma (f\circ s) = 0 \Longleftrightarrow f \Gamma (s) = 
\beta _0(f,s).$ 

(c): If $g \in \mbox{Ann}(s)$, 
then $g\beta_0(f,s) = gf\Gamma (s) = 
\beta _0(g,s)f$, so that $g \in (f:\beta_0(f,s)/X)$. Conversely, if
$g\beta_0(f,s)/X= hf$ for some $h \in R[X]$, then $g\Gamma(s) = 
g\beta_0(f,s)/f = Xh$ and so by Corollary~\ref{1dcharpoly}, 
$g \in \mbox{Ann}(s).$
\end{proof}

The next result is a considerable generalization of Theorem 1 of
Prabhu \& Bose (1982) and Theorem 4.11 of Fitzpatrick \& Norton (1995), 
both of which use left--shifting. 

\begin{lemma} \label{genlemma} Let $R$ be a factorial domain, 
$s \in S^1(R)^-, \ f \in \mbox{Ann}(s)$ with $\delta f \geq 1$. Put  
$d = \gcd (f,\beta _0(f,s)/X)$ and $g = f/d$. Then

(a) $g \in \mbox{Ann}(s)$

(b) $\beta _0(g,s) = \beta _0(f,s)/d$

(c) $\gcd(g, \beta _0(g,s)/X) = 1$

(d) if $d = 1$ and $h \in \mbox{Ann}(s)$ , then $\beta_0(f,s) | \beta_0(h,s)$
and $h = ( \beta _0(h,s)/\beta _0(f,s)) \cdot f$.
\end{lemma}

\begin{proof}
(a) Since $f \in \mbox{Ann}(s), f \Gamma (s) = \beta _0(f,s)$ 
and so $g \Gamma (s) = X\ (\beta _0(f,s)/X)/d \in XR[X]$. By 
Corollary~\ref{1dcharpoly}, $g \in \mbox{Ann}(s)$.

(b) $\beta _0(f,s)/d = f \Gamma (s)/d = g \Gamma (s) = \beta _0(g,s)$ by 
part (a) since $g \in \mbox{Ann}(s)$.

(c) By part (b), $\gcd(g,\beta _0(g,s)/X) = \gcd(g,(\beta _0(f,s)/X)/d) = 1$.

(d) From Proposition~\ref{1dgamma}, $h \beta _0(f,s)/X = f \beta _0(h,s)/X$ 
and since $d = 1,\ \beta _0(f,s)/X$ divides $\beta _0(h,s)/X.$
\end{proof}

\begin{remark} On the other hand, if $g \in \mbox{Ann}(s)\setminus\{0 \}$ has 
minimal degree, then $\gcd(g,\beta _0(g,s)/X) \in R$ --- otherwise we can 
replace $g$ by a characteristic polynomial of smaller degree by part (a) 
of the previous result. If $g$ is also primitive, then 
since $d=\gcd(g,\beta_0(g,s)/X) \in R, \ d | g_0$, and inductively, 
$d$ divides all the coefficients of $g$. Thus $d = 1$. 
\end{remark}

\begin{theorem} Let $R$ be a factorial domain and Ann($s) \neq (0).$ Then

(a) $\mbox{Ann}(s)$ is generated by a {\em primitive} polynomial $g$, say.

(b) $g$ is unique up to a unit of $R$, and $\gcd(g, \beta_0(g,s)/X) = 1.$

(c) $\delta g = \min \{\delta f : f \in \mbox{Ann}(s)\setminus\{0 \} \}$.
\end{theorem}

\begin{proof} We need only prove (a). For any $f \in 
\mbox{Ann}(s)$ with $\delta f \geq 1, \ g= f/ \gcd(f,\beta _0(f,s)/X)$ 
generates $\mbox{Ann}(s)$ by Lemma~\ref{genlemma}, and 
we may take $g$ to be primitive.
\end{proof}

\begin{definition} If $R$ is a factorial domain and Ann$(s) \neq (0)$, 
we denote the {\em primitive}
generator of Ann($s$) (which is unique up to a unit of $R$) by $\gamma(s).$
\end{definition}

We remark that $\gamma (s)$ may be computed using the minimal realization
algorithm of Norton (1995a); see {\em loc. cit.} Corollary 3.28. 
For an alternative approach to $\gamma(s)$ using polynomial remainder
sequences, see Fitzpatrick \& Norton (1995), Section 4.5.

\subsection{Associated sequences} 

\begin{definition} For $n \geq 2$ and $\BFi \ll \BFn$, let
$s^{(\hbox{\bf \scriptsize i},\ast)} 
\in 
S^{\mid\hbox{\bf \scriptsize i}\mid}(R[[X_{\hbox{\bf \scriptsize i}'}]])$ 
be given by 
$$s^{(\hbox{\bf \scriptsize i},\ast)}_{\ a} = 
\sum_{a' \in -{\Bbb N}^{\hbox{\bf \scriptsize i}'}} 
s_{(a,a')} X^{\ a'}_{\hbox{\bf \scriptsize i}'}$$
for $a \in -{\Bbb N}^{\hbox{\bf \scriptsize i}}$.
\end{definition}

Thus for $\BFi \ll \BFn$, the superscript 
$^{(\hbox{\bf \scriptsize i},\ast)}$ is used to suggest 
{\em summation over} $\hbox{\bf i}'$  
(whereas for $a \in -{\Bbb N}^{\hbox{\bf \scriptsize i}}$, the 
superscript $^{(a)}$ is used to denote a {\em section}). 
Special cases of these associated
sequences were used in Fitzpatrick \& Norton (1990).

For $\BFi \ll \BFn$, we let $R[X_{\sBFi}]$ act on $n$--D sequences via:

$$\left(\sum_{a \in {\scriptsize Supp(f)}} 
f_a X_{\hbox{\bf \scriptsize i}}^a \circ s \right) _{b} 
=\sum_{a\in {\scriptsize Supp(f)}} f_a \ s_{b-(a,\hbox{\bf \scriptsize 0}')}$$ 
where $b \leq \BFZ$ and $\BFZ ' \in {\Bbb N}^{\sBFi '}$.  

\begin{lemma} \label{sectionlemma} 
Let $n \geq 2, \BFi \ll \BFn, f \in R[X_{\hbox{\bf \scriptsize i}}], 
f' \in R[X_{\hbox{\bf \scriptsize i}'}]$ 
and $a \in -{\Bbb N}^{\hbox{\bf \scriptsize i}}, 
a' \in -{\Bbb N}^{\hbox{\bf \scriptsize i}'}$. Then 

(a) $(f'\circ s^{(a)})_{a'} = (f'\circ s)_{(a,a')}$

(b) $s^{(\hbox{\bf \scriptsize i},\ast)}_a 
= \G(s^{(a)})(X_{\hbox{\bf \scriptsize i}'}^{-\sBFone})$

(c) $\G(f\circ s^{(\hbox{\bf \scriptsize i},\ast)})
(X_{\hbox{\bf \scriptsize i}}^{-\sBFone}) 
= \G(f\circ s)(\BFX^{-\sBFone})$.
\end{lemma}

\begin{proof}
(a)
\begin{eqnarray*}
(f'\circ s^{(a)})_{a'}\\
& = & \sum_{d'\in \mbox{{\scriptsize Supp}}(f')} f'_{d'} s^{(a)}_{a'-d'}
= \sum_{d'\in \mbox{\scriptsize Supp}(f')} f'_{d'} s_{(a'-d',a)} \\
& = & \sum_{d'\in \mbox{{\scriptsize Supp}}(f')} 
f'_{d'} s_{(a',a)-(d',\hbox{\bf \scriptsize 0})} 
= (f'\circ s)_{(a',a)}. 
\end{eqnarray*}
(b), (c) These are also easy consequences of the definitions.
\end{proof}

The characteristic ideals of associated sequences are related as follows:

\begin{proposition} \label{annassoc}
Let $n \geq 2$ and $\BFi \ll \BFn$. Then
$$\bigcap_{a'\in -{\Bbb N}^{\hbox{\bf \scriptsize i}'}} \mbox{Ann}(s^{(a')}) 
= \mbox{Ann}(s) \cap R[X_{\hbox{\bf \scriptsize i}}] = 
\mbox{Ann}(s^{(\hbox{\bf \scriptsize i},\ast)}).$$
\end{proposition}

\begin{proof} 
Let $f \in R[X_{\hbox{\bf \scriptsize i}}]$. By Lemma~\ref{sectionlemma}(a), 
$$\forall (a' \in -{\Bbb N}^{\hbox{\bf \scriptsize i}'})( f\circ s^{(a')} = 0)
\Longleftrightarrow \forall (b \in -{\Bbb N}^n) ((f\circ s)_b = 0)
\Longleftrightarrow f \in \mbox{Ann}(s).$$ 
To prove the second equality, note that 
$f\circ s = 0 \Longleftrightarrow 
f\circ s^{(\hbox{\bf \scriptsize i},\ast)} = 0$ 
by Lemma~\ref{sectionlemma}(c). 
\end{proof}

\subsection{Some properties of EVR sequences}

It follows from the previous subsection that for $n \geq 2$ and $i \in \BFn$, 
there are $n$ associated sequences $s^{(i,\ast)} = s^{(\{i\},\ast)} \in 
S^1(R[[\widehat{\bf X}_i^{-\hbox{\bf \scriptsize 1}}]])^{-}$ given by 
$$s^{(i,\ast)}_{\ j} 
= \sum_{j' \in -{\Bbb N}^{ \hbox{\bf \scriptsize n}\setminus\{i \}} } 
s_{(j,j')} \widehat{\bf X}_i^{\ j'}$$ for $j \leq 0,\ 
\mbox{Ann}(s^{(i,\ast)}) = \mbox{Ann}(s) \cap 
R[X_i]$ and $\Gamma (f_i\circ s^{(i,\ast)})(X_i^{-1}) = 
\Gamma (f_i\circ s)(\BFX^{-\hbox{\bf \scriptsize 1}}).$

The next result is a considerable generalization of 
Prabhu \& Bose (1982), Theorem 1 
and Fitzpatrick \& Norton (1990), Corollary 4.2:

\begin{theorem} \label{evrgamma} Let $f_i \in R[X_i]$ with 
$\delta f_i \geq 1$
for all $i \in \BFn$ and let $f = \prod_{i=1}^n f_i$. 
The following are equivalent:

(a) for all $i \in \BFn, \ f_i \in \mbox{Ann}(s)$ 

(b) $f\Gamma (s) \in  \BFX P_n$

(c) $f\Gamma (s) = \beta _0(f,s)$.
\end{theorem}

\begin{proof} 
(a) $\Longrightarrow$ (b): 
For all $i \in \BFn$ and $a_i \leq 0$, the coefficient of $X_i^{a_i}$ in 
$f\Gamma(s)$ is
$$(f\Gamma (s))_{(a_i, \hbox{\bf \scriptsize 0}')} = 
(f\circ s)_{(a_i,\hbox{\bf \scriptsize 0}')}$$ 
which is zero since $f$ is a multiple of $f_i \in \mbox{Ann}(s).$
Hence Supp$(f\Gamma (s)) \subseteq [\BFone, \delta f].$

(b) $\Longrightarrow$ (a): Fix $i \in \BFn$. Put $t = s^{(i,\ast)}$ 
and let $f\Gamma(s) = \BFX g$, where $g \in P_n.$ 
By Lemma~\ref{sectionlemma}, 
$\Gamma(t)(X_i^{-1}) = \Gamma(s)(\BFX^{-\hbox{\bf \scriptsize 1}}).$
Thus 
$$\delta_i (f_i \Gamma (t))  = \delta_i (f \Gamma (t)) 
= \delta_i (f\Gamma(s)) = \delta_i \BFX g \geq 1,$$
whereas for $j \neq i,$
$$\delta_j (f_i\Gamma(t)) \leq \delta_jf_i + \delta_j\Gamma(t) 
\leq \delta_j \Gamma(t) \leq 0$$
i.e. $f_i\Gamma(t) \in
X_i(R[[\widehat{\bf X}_i^{-\hbox{\bf \scriptsize 1}}]])[X_i].$ 
Thus by Corollary~\ref{1dcharpoly}, $f_i \in \mbox{Ann}(t)$
and by Lemma~\ref{sectionlemma} again, $\Gamma 
(f_i\circ s) = \Gamma (f_i\circ t) = 0$ i.e. $f_i \in \mbox{Ann}(s)$.

(b) $\Longrightarrow$ (c): Let $f\Gamma (s) = \BFX g \in P_n$ where 
$\delta g \leq \delta f-\BFone$. Since (b) $\Longleftrightarrow$ (a), 
$\Gamma(f\circ s) = 0$ and so
$$\BFX g - \beta_0(f,s) = f\Gamma (s) - \beta_0(f,s) = 
\sum_{k=1}^{2^n-2}\beta_k(f,s).$$ 
The support of the left--hand side is contained in $[\BFone,\delta f]$, 
whereas the support of the right--hand side is 
contained in $(-\infty,\delta f]\setminus[\BFone,\delta f]$ by construction. 
Therefore both sides vanish and $f\Gamma (s) = \beta _0(f,s).$
\end{proof}

We now discuss $\mbox{Ann}(s) \cap R[X_i]$, 
which is an ideal of $R[X_i]$. So when $R$ is a field, 
$\mbox{Ann}(s) \cap R[X_i]$ is 
principal and has a {\em monic} generator. 
We will see that this is also true more generally. 

For $n=1$, the following notion was introduced in Fitzpatrick \& Norton
(1995):

\begin{definition} A factorial domain $R$ will be called {\em potential} 
if for all $n \geq 1,\  S_n$ is factorial.
\end{definition}

In fact, for the known examples $R$ for which $R[[X]]$ is factorial, $S_n$ is also 
factorial. Principal ideal domains, $R[\BFX]$  (where $R$ is a field
or a discrete valuation ring) are potential domains, but
not every factorial domain is potential: 
see Bourbaki (1972), Exercise 8(c), p.566. We refer the reader to 
Bourbaki (1972), Proposition 8, p.511 and Exercise 9(c), p.566 and to Fossum 
(1973) for more details.

\begin{lemma} \label{monicgen} Let $R$ be factorial and $s \in S^1(R)^-$. 
If $R[[X]]$ is factorial, we can assume that $\gamma(s)$ is monic.
\end{lemma}

\begin{proof}
Ann$(s) \cong Ann(s^-)$, and so the result follows from Fitzpatrick
\& Norton (1995), Corollary 4.8.
\end{proof}

\begin{theorem} \label{potential}
If $R$ potential and $s$ is EVR, then for all 
$i \in \hbox{\bf n}, \mbox{Ann}(s) \cap R[X_i]$ is principal and 
has a unique {\em monic} generator.
\end{theorem}

\begin{proof} We can assume that $n \geq 2$. Fix $i$ and let $U = 
R[[\widehat{\bf X}_i^{-\hbox{\bf \scriptsize 1}}]]$ and
$t = s^{(i,\ast)} \in S^1(U)^-$. By Proposition~\ref{annassoc}, 
$\mbox{Ann}(s) \cap R[X_i] = \mbox{Ann}(t)$. Also, $U$ and 
$U[[X_i^{-1}]] = S_n$ are factorial since $R$ is potential. 
Thus the result follows from Lemma~\ref{monicgen}.
\end{proof}

The next result simplifies and extends 
Fitzpatrick \& Norton (1990), Theorem 4.3:  

\begin{theorem} \label{Annsi} 
Let $R$ be a potential domain and for all $i \in \BFn$, let $f_i \in 
\mbox{Ann}(s) \cap R[X_i]$ with $\delta f_i \geq 1$. Put 
$f = \prod_{i=1}^n f_i$ and $d_i = \gcd(f_i, \beta _0(f,s)/X_i )$. 
Then for all $i \in \BFn$, there is a unit $u_i$ of $R$ such that
$u_i f_i/d_i$ is the monic generator of  Ann$(s) \cap R[X_i].$
\end{theorem}

\begin{proof} For $i \in \BFn$, let $\gamma _i$ be the monic generator of 
$\mbox{Ann}(s) \cap R[X_i] \mbox{ and } 
\gamma = \prod_{i=1}^n \gamma _i.$
By Theorem~\ref{evrgamma}, 
$$\gamma \beta _0(f,s) = \gamma f \Gamma (s) = f \beta _0(\gamma ,s).$$ 
Now $\gcd(\gamma ,\beta _0(\gamma ,s)/\BFX) 
\in R$, otherwise we could replace some $\gamma _j$ by a polynomial of smaller 
degree. Since $P_n$ is factorial, $\gamma $ is primitive by Gauss' 
Lemma and so $\gcd(\gamma ,\beta _0(\gamma ,s)/\BFX) = 1$. Hence there is a 
unit $u$ of $R$ with 
$\gamma = u f / \gcd(f,\beta _0(f,s)/\BFX)$. Also, 
$$\gcd(f,\beta _0(f,s)/\BFX) = \gcd( \prod f_i, \beta _0(f,s)/\BFX ) 
= \prod \gcd(f_i,\beta _0(f,s)/X_i) = \prod d_i$$ 
since $f_i \in R[X_i]$. Thus 
$\prod \gamma _i= u \prod (f_i/ d_i)$ and since $P_n$ is factorial, 
$\gamma _i = u_i f_i / d_i$ for some unit $u_i$ of $R$.
\end{proof}

We now combine Theorems~\ref{evrgamma} and ~\ref{Annsi} to 
characterize generating functions of EVR lrs:

\begin{corollary} \label{Annsigen} Let $f_i \in R[X_i]$ 
with $\delta f_i \geq 1$ 
for $i \in \BFn$, and put $f = \prod_{i=1}^nf_i$. 

(a) If $R$ is potential and $f_i \in \mbox{Ann}(s)$ for all 
$i \in \BFn$, then $\Gamma (s) = \beta _0(f,s)/f$ 
and we can assume that for all $i \in \BFn, f_i$ is monic, 
$\gcd(f_i,\beta _0(f,s)/X_i) = 1$ and $\mbox{Ann}(s) \cap R[X_i] = (f_i)$. 

(b) If for all $i \in \BFn, f_i$ is monic, $g \in R[\BFX], 
\delta g \leq \delta f-\BFone$ and $\gcd(f_i, g) = 1$, then $s$ defined by 
$\Gamma (s) = \BFX g/f$ is an (EVR) lrs with $\BFX g = \beta _0(f,s)$ and 
$\mbox{Ann}(s) \cap R[X_i] = (f_i)$.
\end{corollary}

\begin{proof} For each $i \in \BFn, f_i$ can be chosen as the monic 
generator of $\mbox{Ann}(s) \cap R[X_i] = \mbox{Ann}(s^{(i,\ast)})$ by 
Theorem~\ref{potential}. Since 
$$f_i\Gamma (s^{(i,\ast)}) = f_i\Gamma (s) = \beta _0(f,s) / 
\prod_{j\neq i} f_j$$ 
we must have $d_i = \gcd(f_i,\beta _0(f,s)/X_i) \in R$ 
--- for otherwise by Lemma~\ref{genlemma}, we can replace $f_i$ by a 
characteristic polynomial of smaller degree. 
Since $d_i$ divides each coefficient of $f_i$, which is monic, $d_i = 1$.

For the converse, we have $\BFX g = \beta _0(f,s)$ from 
Theorem~\ref{evrgamma}.
Thus $\gcd(f_i, \beta _0(f,s)/X_i) = 1$ and by Theorem~\ref{Annsi}, 
$\mbox{Ann}(s) \cap R[X_i] = (f_i).$
\end{proof}

For completeness, we give an alternative way of finding the generator of 
$\mbox{Ann}(s) \cap R[X_i]$ (Corollary ~\ref{Annlcm} below), generalizing 
Fitzpatrick \& Norton (1990), Theorem 4.10.

\begin{lemma} \label{intersect} 
Let $n \geq 2$, $R$ be a domain, let $f_i \in \mbox{Ann}(s)\cap R[X_i]$ 
with $\delta f_i \geq 1$
for all $i \in \hbox{\bf n}$, and put $\Delta'_i
= \delta (\prod_{j=1,j\neq i}^n f_j) - \hbox{\bf 1}' 
\in {\Bbb N}^{\hbox{\bf \scriptsize n}\setminus\{i \}}$. Then 
$$\bigcap_{a' \in {-\Bbb N}^{\hbox{\bf \scriptsize n} \setminus \{i \}}} 
\mbox{Ann}(s^{(a')}) = 
\bigcap_{-\Delta'_i \leq a'\leq \hbox{\bf \scriptsize 0}'} \mbox{Ann}(s^{(a')}).$$
\end{lemma}

\begin{proof} By definition, $s^{(a')} \in S^1(R)^-$ if $a' \in 
-{\Bbb N}^{\hbox{\bf \scriptsize n}\setminus\{i \}}$. We show that if 
$g \in R[X_i]$ satisfies $g\circ s^{(a')} = 0$ for 
$-\Delta'_i \leq a' \leq \bf 0'$, then $g\circ s^{(a')} = 0$ 
for all $a'$. Let $i=1$ and let $\lambda_2$ be the leading coefficient of 
$f_2$. Suppose first that 
$a' = (-\delta _2f_2, -\delta _3f_3+1, \ldots ,-\delta _nf_n+1)$. 
Since $f_2 \in \mbox{Ann}(s)$, we can write 
$$\lambda_2 s_{(a,a')} = 
- \sum_{b'=0}^{\delta _2f_2-1} (f_2)_{b'} s_{(a,-b',c')}$$ 
for all $a \leq 0$, where $c' = (-\delta _3f_3+1, \ldots 
,-\delta _nf_n+1)$. Then for $a \leq 0,$
$$(\lambda_2 g\circ s^{(a')})_a = \sum_{c=0}^{\delta g} g_c \lambda_2 
s_{(a-c,a')} 
= - \sum_{b'=0}^{\delta _2f_2-1} (f_2)_{b'} \sum_{c=0}^{\delta g} g_c 
s_{(a-c,-b',c')}$$ 
and $\sum_{c=0}^{\delta g} g_c s_{(a-c,-b',c')} = (g\circ 
s^{(-b',c')})_a = 0$ since $\Delta'_1 \leq (-b',c') \leq \BFZ'$. 
Thus for $a' = 
(-\delta _2f_2, -\delta _3f_3+1, \ldots ,-\delta _nf_n+1), \
\lambda_2 g\circ s^{(a')} = 0$. The argument 
is similar for $a' = (a_2, -\delta _3f_3+1, \ldots ,-\delta _nf_n+1), 
\ a_2 < -\delta _2f_2$; in 
this case, $\lambda_2^{-a_2-\delta _2f_2} g\circ s^{(a')} = 0$. 
Since $\lambda_2 \in R$ and $R$ 
is a domain, $g\circ s^{(a')} = 0$ for all $a_2 \leq 0$. 

Suppose inductively 
that the result is true for 
$a' = (a_2, a_3, \ldots , a_{k-1}, -\delta _kf_k+1,\ldots , 
-\delta _nf_n+1)$ and $a' = (a_2, a_3, \ldots , a_k, -\delta _{k+1}f_{k+1}+1, \ldots , 
-\delta _nf_n+1), a_k \leq -\delta f_k$. 
An analogous argument using $f_k \in \mbox{Ann}(s) 
\cap R[X_k]$ shows that $g\circ s^{(a')} = 0$. 
By induction $g\circ s^{(a')} = 
0$ for all $a' \in -{\Bbb N}^{\hbox{\bf \scriptsize n}\setminus\{1 \}}$. 
It is clear that the same argument 
applies to any $i \in \hbox{\bf n}$ and this completes the proof.
\end{proof}

\begin{theorem} \label{Annlcm}  
Let $n \geq 2$ and $R$ be potential. Suppose that for 
all $i \in \hbox{\bf n}, f_i \in \mbox{Ann}(s)\cap R[X_i], 
\ \delta f_i \geq 1$ and 
let $\Delta'_i= \delta (\prod_{j=1,j\neq i}^nf_j) - \hbox{\bf 1}' 
\in {\Bbb N}^{\hbox{\bf \scriptsize n}\setminus\{i \}}$. Then 
for all $i \in \hbox{\bf n}$, 
$$\mbox{Ann}(s) \cap R[X_i] = 
\left( \lcm\{\gamma(s^{(a')}) : -\Delta'_i \leq a' \leq 
\hbox{\bf 0}' \} \right)$$ 
\end{theorem}

\begin{proof} Fix $i \in \hbox{\bf n}$. For $a' \in 
-{\Bbb N}^{\hbox{\bf \scriptsize n}\setminus\{i\}}, \gamma(s^{(a')}) 
\in R[X_i]$ is well--defined (and indeed monic) by Lemma~\ref{monicgen}. 
By Proposition~\ref{annassoc} and 
Lemma~\ref{intersect} 
$$\mbox{Ann}(s) \cap R[X_i]=\bigcap_{-\Delta'_i\leq a'\leq \bf 0'}\mbox{Ann}(s^{(a')}).$$ 
Since $R[X_i]$ is 
factorial, the intersection is non--empty and is generated by 
$\mbox{lcm}\{\gamma(s^{(a')}) : -\Delta'_i \leq a' \leq \hbox{\bf 0}' \}$.
\end{proof}

The least common multiple of Theorem~\ref{Annlcm} may be computed using 
Algorithm 3.30 of Norton (1995a). 
If $R$ is a field, it may also be computed using
the ``Fundamental Iterative Algorithm'' of Feng \& Tzeng (1989).

We conclude this subsection by generalizing a result of Cerlienco \& Piras 
(1991) to sequences over domains. 

\begin{definition}  
We say that an ideal $I$ of $P_n$ is {\em cofinite} 
if $P_n/I$ is a finitely generated $R$--module and that an $n$--D lrs over 
$R$ is {\em cofinite} if its characteristic ideal contains a cofinite ideal.
\end{definition}

It is well known that if $R$ is a field, then a cofinite ideal in 
$R[X_1,\ldots,X_n]$ has a finite zero set. Conversely, if $R$ is an 
algebraically closed field, then an ideal in $R[X_1,\ldots,X_n]$ with a 
finite zero set is cofinite.

\begin{theorem}  \label{cofinite}
If $s$ is a cofinite lrs over $R$, then $s$ is EVR. 
Conversely, if $R$ is potential and $s$ is EVR, then $s$ is cofinite. 
\end{theorem}

\begin{proof} Suppose that $\mbox{Ann}(s) \cap R[X_i] = \{0 \}$ for some $i$. Then 
$$R[X_i] = 
R[X_i] / (\mbox{Ann}(s) \cap R[X_i]) 
\cong \left(R[X_i] + \mbox{Ann}(s) \right)/ \mbox{Ann}(s) 
\subseteq P_n/\mbox{Ann}(s)$$
and so $P_n / \mbox{Ann}(s)$ is not finitely generated. Thus $\mbox{Ann}(s)$ 
contains no cofinite ideals $I$ (otherwise there would be a map of 
$P_n/I$ {\em onto} $P_n / \mbox{Ann}(s)$, 
implying that $P_n/\mbox{Ann}(s)$ is finitely generated) 
i.e. $s$ is not cofinite. 

If $R$ is potential and $s$ is EVR, then by Theorem~\ref{potential} 
we may assume that the 
$f_i\in \mbox{Ann}(s)$ are monic, and set $F =(f_1,f_2,\ldots ,f_n)$. 
Then the 
monomials $\BFX^a \bmod F$ for $a \leq (\delta f_1-1, \delta f_2-1,\ldots , 
\delta f_n-1)$ 
generate $P_n / F$. Thus $F \subseteq \mbox{Ann}(s)$ is cofinite i.e. 
$s$ is cofinite.
\end{proof}

In particular, if $R$ is a field, the two notions coincide (cf. Cerlienco \& 
Piras (1991), Section 4). 

\subsection{Ann(s) as Ideal Quotient}
 
We generalize Proposition~\ref{1dgamma}(c)
and Fitzpatrick \& Norton (1990), Theorem 5.1 to $n$--D sequences over $R$. 
(The first part of the proof generalizes the proof for
the case $n=1$ and the second part generalizes the minimal
counterexample idea of Fitzpatrick \& Norton (1990), {\em loc. cit.})

\begin{theorem} \label{main} If $f_i \in 
\mbox{Ann}(s)\cap R[X_i]$, with $\delta f_i \geq 1$
for all $i \in \BFn$ and $f = 
\prod_{i=1}^n f_i$, then 
$$\mbox{Ann}(s) = \left(\sum_{i=1}^n(f_i) : \beta _0(f, s)/\BFX\right).$$
\end{theorem}

\begin{proof} By Proposition~\ref{1dgamma}(c), we can assume that $n \geq 2.$

$\supseteq$: Suppose that $g \in P_n\setminus\{0 \}$ satisfies $g \beta 
_0(f,s) = \BFX \sum_{i=1}^n f_i u_i$ for some $u_i \in P_n$. By 
Theorem~\ref{evrgamma}, 
$$g \Gamma (s)= g \beta _0(f,s) / f = 
\left(\BFX \sum_{i=1}^n f_i u_i \right) / f 
= \BFX \sum_{i=1}^n \left( u_i \ / \prod_{j=1,j\neq i}^n f_j \right)$$ 
and so for $a \leq \BFZ$,
$$(g\circ s)_a = \sum_{i=1}^n 
\left(u_i \ / \prod_{j=1,j\neq i}^n f_j \right)_{a-\hbox{\bf \scriptsize 1}}=0$$ 
since each $u_i \in P_n$ i.e. $g \in \mbox{Ann}(s)$.

$\subseteq$: Let $g \in \mbox{Ann}(s)$ and $h = g \Gamma (s) \in L_n$. 
By Proposition~\ref{charpoly}(c), 
$\mbox{Supp}(h) \subseteq (\BFZ,\delta g]$. Let $\wp 
= \{\wp_i : 1 \leq i \leq n \}$ be the following partition of 
$(\BFZ,\delta g]$ : 
$$\wp_1 = \{ a \in (\BFZ,\delta g] : a_1 > 0 \} 
\mbox{ and } \wp_i 
= \{ a \in (\BFZ,\delta g] \setminus \bigcup _{j=1}^{i-1} \wp_j : a_i > 0 \}$$ 
if $2 \leq i \leq n$. It is 
clear that $\wp$ is a partition of $(\BFZ,\delta g]$ and that for $1 \leq i 
\leq n, \wp_i = \{ a \in (\BFZ,\delta g] : a_j \leq 0$ for $1 \leq j \leq 
i-1$ and $a_i > 0 \}$. Let $h_i = h|\wp_i$  and
$p_i = h_i f/f_i$ for $i \in \BFn$. 
Then $h_i \in X_i 
R((\widehat{\bf X}_i^{-\hbox{\bf \scriptsize 1}}))[X_i]$ for all $i \in \BFn$
and
$$g \beta _0(f,s) = g f \Gamma (s) = hf 
= \sum_{i=1}^n h_if = \sum_{i=1}^np_i f_i.$$ 
So it suffices to show that $p_i \in \BFX P_n$ for all $i \in \BFn$.

Suppose that for some $i \in \BFn, p_i \not\in \BFX P_n$, and let $k$ be 
the smallest such integer. Now $p_k/X_k$ is a polynomial in $X_k$, 
and so for some $r \geq 1$, the coefficient of $X_k^r$ in $p_k$ 
is not in $R[\widehat{\bf X}_k]$. 
Letting $d = \delta f_k$, the coefficient of $X_k^{d+r}$ in $h_kf = 
p_kf_k$ cannot therefore be in $R[\widehat{\bf X}_k]$. But 
$$\sum_{i=1}^n h_i f = g\beta _0(f,s) \in \BFX P_n$$ 
and so the 
coefficient of $X_k^{d+r}$ in $h_kf$ must cancel with other 
coefficients of $X_k^{d+r}$ in the $h_if$, for one or more $i \neq k$. 
If $1 \leq i < k$, the cancellation cannot take place by choice of $k$. If $k 
< i \leq n$, the exponent of $X_k$ in $fh_i$ is at most $d$ by 
construction of $h_i$. 
Thus cancellation of the coefficient of $X_k^{d+r}$ in 
$h_kf$ cannot take place and we conclude that $p_i \in \BFX P_n$ for 
all $i \in \BFn$, as required.
\end{proof}

\begin{example} Let $s$ be as in Example~\ref{exampleab}(a). 
We know that $f_1 = X_1$ and $f_2 = X_2-1$ belong to $\mbox{Ann}(s)$ and 
$\Gamma (s) = 1/(X_2-1), \ \beta _0(f_1f_2,s) = X_1X_2$. 
From Theorem~\ref{main}, $\mbox{Ann}(s) = ((f_1,f_2) : 1) = 
(f_1,f_2) = (X_1, X_2-1)$. We remark that Fitzpatrick \& Norton (1990), 
Theorem 5.2 gives $\mbox{Ann}(s) = (X_2-1)$ if $R$ is a field. 
\end{example}

\begin{example} Let $s$ satisfy $s_{(i,0)}=1$ for $i \leq 0$, $s_{(0,j)}
= 1$ for $j \leq 0$ and $s_{(i,j)}=0$ otherwise. Then 
$\Gamma(s) = (X_1X_2-1)/(X_1-1)(X_2-1)$, so that by Theorem~\ref{main},
Ann$(s) = (X_1(X_1-1), X_2(X_2-1) : X_1X_2-1).$
\end{example}

Combining Theorems~\ref{Annsi}, \ref{Annlcm} and 
~\ref{main} now yields an algorithm for computing a basis for $\mbox{Ann}(s)$:

\begin{algorithm} \label{Anness} (cf. Fitzpatrick \& Norton (1990), 
Algorithm 3.8).

 Input: An EVR $s \in S^n(R)^-, R$ a potential domain.

 Output: A basis for $\mbox{Ann}(s)$.

1. For each $i \in \BFn$, find the monic generator 
$\gamma _i$ of $\mbox{Ann}(s) \cap R[X_i]$ using Theorem~\ref{Annsi} 
or Theorem~\ref{Annlcm}. 

2. Compute $\gamma = \prod_{i=1}^n \gamma _i$ and 
$\beta _0(\gamma ,s)/ \BFX$.

3. Find a basis for the ideal quotient $( \sum_{i=1}^n (\gamma _i) : \beta 
_0(\gamma ,s)/\BFX ).$ 
\end{algorithm}

\begin{remark} Over a field ${\Bbb F}$, Step 3 may be done in several 
ways: 

(i) using the basis $\{\BFX^m : m \leq \delta \gamma -\BFone \}$ for ${\Bbb 
F}[\BFX]/(\gamma _1,\gamma _2,\ldots ,\gamma _n)$, find an ${\Bbb F}$-basis 
for solutions $f$ of the homogeneous equation 
$$f \beta _0(\gamma , s)/\BFX + \sum_{i=1}^n \gamma _i u_i = 0$$ 
where $u_i \in {\Bbb F}[\BFX]$.
Augmenting these solutions by $\{\gamma _1, \gamma _2,\ldots ,\gamma _n \}$ 
yields an ${\Bbb F}$--basis for $\mbox{Ann}(s)$. 

This approach is an application of Buchberger (1985), Method 6.7, which 
applies since $\{\gamma _1, \gamma _2, \ldots ,\gamma _n \}$ is trivially a 
(reduced) groebner basis for any term order. In this case, 
Algorithm~\ref{Anness} has 
complexity $O(\prod_{i=1}^n \delta \gamma _i^3)$. The ${\Bbb F}$--basis found 
can of course be converted to a reduced groebner basis with respect to any 
given term order, if desired.

(ii) Cox et al. (1991), Sections 3.3, 3.4 reduce the computation of a basis for 
an ideal quotient to computing a groebner basis (with respect to the 
lexicographic term order) of an intersection of ideals. In this way, we obtain 
a groebner basis for $\mbox{Ann}(s)$. This approach was applied to
computing an ideal basis of a $2$--D cyclic code and its dual 
in Norton (1995b).
\end{remark}

{\sc Acknowledgements}: The author would like to thank Aidan Schofield for a 
useful conversation, the U.K. Science and Engineering Research Council for 
financial support and the referees for helpful comments and suggestions. 
Thanks also to Guy Chass\'{e} for a copy of Ferrand (1988).

{\sc Note added in proof}: 
This paper combines two papers ``On $n$--dimensional
sequences I,II'' that were submitted to this journal. The first paper 
described a border partition for left--shifting, gave an inductive proof of 
Corollary~\ref{leftdecomp} and some applications of it.  
Lemma~\ref{shortproof} (and right--shifting) 
appeared in the second paper. The author also
suggested a new proof of Corollary~\ref{leftdecomp} 
(via Lemma~\ref{shortproof} and a direct, rather than inductive proof
of Theorem~\ref{borderseries}) 
in the second paper. This revised version adopts the new approach.

\begin{center}{\bf References}
\end{center}

Bourbaki, N. (1972). {\em Commutative Algebra\/}. Hermann. 

Bozta\c{s}, S., Hammons, A.R., Kumar, P.V. (1992).
4--Phase sequences with near--optimum correlation properties. 
{\em IEEE Trans. Information Theory\/} {\bf  38}, 1101--1113.

Buchberger, B. (1985). 
Groebner bases: an algorithmic method in polynomial ideal theory. 
{\em Multidimensional Systems Theory, N.K. Bose (ed.)\/},
184--232. Reidel, Dordrecht.

Cerlienco, L., Mignotte, M., Piras, F. (1987). 
Suites R\'{e}currentes Lin\'{e}aires. 
{\em L'Enseignement Math\-\'{e}matiques\/} {\bf 33}, 67--108.

Cerlienco, L., Piras, F. (1991). 
On the continuous dual of a polynomial bialgebra. 
{\em Communications in Algebra\/} {\bf 19}, 2707--2727.

Chabanne, H., Norton, G.H. (1994).
The $n$--dimensional key equation and a decoding application. 
{\em IEEE Trans. Information Theory\/} {\bf 40}, 200 -- 203.


Cox, D., Little, J. and O'Shea, D. (1991).
{\em Ideals, Varieties and Algorithms\/}. 
Springer.

Fan, P.Z., Darnell,M. (1994). Maximal length sequences over Gaussian integers.
{\em Electronic Letters\/} {\bf 30}, 1286--1287.

Feng, G.L., Tzeng, K.K. (1989). A generalization of the 
Berlekamp--Massey algorithm for multisequence shift register synthesis 
with applications to the decoding of cyclic codes. 
{\em IEEE Trans. Information Theory\/} {\bf 37}, 1274--1287.

Ferrand, D. (1988). {\em Suites R\'{e}currentes\/}.
IRMAR, Universit\'{e} de Rennes.

Fitzpatrick, P., Norton, G.H. (1990). 
Finding a basis for the characteristic ideal of 
an $n$--dimensional linear recurring sequence. 
{\em IEEE Transactions on Information Theory\/} {\bf 36}, 1480--1487.

Fitzpatrick, P., Norton, G.H. (1991). Linear recurring sequences and the 
path weight enumerator of a convolutional code.
{\em Electronic Letters\/} {\bf 27}, 98--99.

Fitzpatrick, P., Norton, G.H. (1995).
The Berlekamp--Massey algorithm and linear 
recurring sequences over a factorial domain.  
{\em J. Applicable Algebra in Engineering, Communications and Computing.\/}
{\bf 6}, 309--323.

Fossum, R. (1973). 
{\em The divisor class group of a Krull domain\/}. 
Springer Ergebnisse {\bf 74}.
                                                                             
Greene, D.H. \& Knuth, D.E. (1982). 
{\em Mathematics for the Analysis of Algorithms, 2nd 
Edition\/}. Birkh\"{a}user, Boston.

Homer, S. and Goldman, J. (1985). 
Doubly periodic sequences and two-dimensional 
recurrences. {\em SIAM J. Alg. Disc. Math.\/} {\bf 6}, 360--370.

Ikai, T., Kosako, H., Kojima, Y. (1975). Two dimensional cyclic 
codes. {\em Electronics and Communications in Japan\/} {\bf 57--A}, 27--35.

Ikai, T., Kosako, H. and Kojima, Y. (1976). Basic theory of 
two--dimensional cyclic codes.  
{\em Electronics and Communications in Japan\/} {\bf 59--A}, 31--47.

Imai, H., Arakaki, M. (1974). A theory of two--dimensional cyclic codes. {\em
National Convention Record of Institute of Electronics and
Communication Engineers, Japan\/}, 1564. In Japanese.

Imai, H. (1977) A theory of two--dimensional codes. {\em Inform. Control\/}
{\bf 34}, 1--21.

Lidl, R., Niederreiter, H. (1983). 
{\em Finite Fields, Encyclopedia of Mathematics and its 
Applications\/} {\bf 20}. Addison--Wesley.

Lin, D. \& Liu, M. (1988). Linear Recurring $m$-Arrays. 
{\em Proceedings of Eurocrypt 88\/}. LNCS {\bf 330}, 351--357. Springer.


Massey, J.L. (1969). Shift register synthesis and BCH decoding. 
{\em IEEE Trans. Information Theory\/} {\bf 15}, 122--127.

McEliece, R. (1987). {\em Finite Fields for Computer Scientists and 
Engineers\/}. 
Kluwer Academic Press. 

Nerode, A. (1958).  Linear automaton transformations. {\em Proc. Amer.
Math. Soc.\/} {\bf 9}, 541--544.

Niederreiter, H. (1988). Sequences with almost perfect linear
complexity profile. In {\em Advances in Cryptology-- Eurocrypt '87\/}. 
LNCS {\bf 304}, 37--51. Springer.

Norton, G.H. (1995a). On the minimal realizations of a finite
sequence. {\em J. Symbolic Computation\/} {\bf 20}, 93--115.

Norton, G.H. (1995b). The solution module (of $n$--dimensional sequences) of 
an ideal containing $(X_1^{M_1}-1, \ldots, X_n^{M_n}-1)$. 
{\em Proc. $\mbox{IV}^{\mbox{th}}$ I.M.A. Conference on Coding and 
Cryptography (P.G. Farrell, ed.)\/}, Formara Ltd., 315--326.

Norton, G.H. (1995c). Encoding an $n$--dimensional cyclic code over a finite
field. Presented at Third International Conference on Finite Fields,
Glasgow, July 1995.

Norton, G.H. (1995d). Some decoding applications of minimal realization.
{\em Proc. $V^{th}$ I.M.A. Conference on Coding and Cryptography\/}
LNCS {\bf 1025}, 53--62. Springer.

Norton, G.H. (1995e). On the minimal realizations of a finite $n$--D sequence. 
In preparation.

Peterson W. W., \& Weldon, E.J. (1972). 
{\em Error Correcting Codes, 2nd Edition\/}. MIT 
Press. Cambridge, Massachusetts.

Prabhu, K.A. \& Bose, N.K. (1982). 
Impulse Response Arrays of Discrete--Space Systems 
over a Finite Field. 
{\em IEEE Trans. Acoustics, Speech and Signal Processing\/} {\bf 30}, 10--18.


Rueppel, R.A. (1986). {\em Analysis and Design of Stream Ciphers\/}. 
Springer.

Sakata, S. (1978). General theory of doubly periodic
arrays over an arbitrary finite field and its applications. 
{\em IEEE Trans. Information Theory\/} {\bf 24}, 719--730.

Sakata, S. (1981). On determining the independent point
set for doubly periodic arrays and encoding two--dimensional cyclic codes
and their duals. {\em IEEE Trans. Information Theory\/} {\bf 27}, 556--565.

Sakata, S. (1988). Finding a minimal set of linear recurring relations
capable of generating a given finite two--dimensional array.
{\em J. Symbolic Computation\/} {\bf 5}, 321--337.

Sakata, S. (1990). 
Extension of the Berlekamp--Massey algorithm to $n$ dimensions. 
{\em Information and Computation \/} {\bf 84}, 207--239.

Welch, L.R., Scholtz, R.A. (1979).
Continued fractions and Berlekamp's algorithm. 
{\em IEEE Trans. Information Theory\/} {\bf 25}, 19--27.

Zierler, N. (1959). Linear recurring sequences. 
{\em J.\ S.I.A.M.\/} {\bf 7}, 31--48.

\vspace{2cm}

email: ghn@maths.uq.edu.au \hfill \today
\end{document}